\documentclass[a4paper,10pt]{article}
\usepackage[utf8]{inputenc}
\usepackage{amssymb, amsmath, bm, amsthm}
\usepackage{enumerate,color}
\usepackage{hyperref}
\usepackage{bbm,constants}
\usepackage{tikz,graphics}
\usepackage{tikz}\usetikzlibrary{patterns}\usetikzlibrary{calc}
\usepackage{pgfplots}
\usepackage{subcaption}
\usepackage{float}
\usepackage{dsfont}
\newcommand{\ctel}[1]{\Cl{#1}}
\newcommand{\cter}[1]{\Cr{#1}}

\newcommand{\keywords}[1]{\textbf{Keywords:} #1}
\newcommand{\classification}[1]{\textbf{AMS subject classification}(2020): #1}
\newcommand{\nabl}{\nabla\!}
\newcommand{\lunk}{\left[}
\newcommand{\runk}{\right]}

\newtheorem{defn}{Definition}[section]
\newtheorem{thm}[defn]{Theorem}

\newtheorem{lemma}[defn]{Lemma}

\allowdisplaybreaks
\begin{document}

\title{Convergence of a Control Volume Finite Element scheme for a cross-diffusion system modeling ion transport}

\author{A. Berrens\thanks{Department of Mathematics, Technical University of Darmstadt, Dolivostr. 15, 64293
Darmstadt, Germany, {\tt berrens@mathematik.tu-darmstadt.de}} and R. Eymard\thanks{Universit\'e Gustave Eiffel, LAMA, (UMR 8050), UPEM, UPEC, CNRS, F-77454, Marne-la-Vallée, France, {\tt robert.eymard@univ-eiffel.fr}}}

\maketitle

\abstract{
An approximation of a system coupling the cross-diffusion of chemical species within a solvent, subjected to an electric field, is obtained through a control volume finite element (CVFE) scheme on general simplicial meshes in two or three space dimensions. The discrete unknowns of the numerical scheme are derived from the chemical potential of the species.
The scheme is designed in order to fulfill entropy inequalities, yielding compactness properties for the discrete solutions and convergence to a weak solution of the continuous problem. Numerical illustrations of the convergence properties are provided in situations where diffusion of ionic species degenerates.
}\medskip\\
\\
\classification{35K65, 35K51, 65M08, 65M12}
\section{Introduction}
The transport of ionic species in an electrically neutral solvent is classically described by the Poisson-Nernst-Planck equations \cite{nernst1888}, which apply Fick's law independently for each species.
However, for transport through nanopores or membranes, this approach neglects the finite size of the ions and the resulting size exclusion effects.
Therefore, the diffusion of each species depends on the other species and Fick's laws cannot be independently applied to each species.
Hence, we consider the cross-diffusion model proposed in \cite{burger2012}, describing the concentrations of ions $u_1,\dots,u_n$ within a solvent with concentration $u_0$ and an electric potential $\phi$, inside the domain $\Omega\subset{\mathbb{R}}^d$ (polygonal if $d=2$, polyhedral if $d=3$) up to time $T>0$:
\begin{subequations}
\begin{equation}\label{eq:uicont}
	\partial_t u_{i} - D_i{\rm div} (u_{0} \nabl u_{i} - u_{i} \nabl u_{0} + u_{0}u_{i} \beta z_i \nabl \phi) = 0,\quad i=1,\ldots,n,\quad\text{ in }(0,T)\times\Omega,
\end{equation}
with
\begin{equation}\label{eq:sumicont}
\sum_{i=0}^n u_{i} = 1,
\end{equation}
under some initial conditions:
\begin{equation}\label{eq:inicont}
u_i(0,\,\cdot\,) = u_i^{\rm ini}\hbox{ in }\Omega,\ i=0,\ldots,n, \hbox{ with }M_i = \int_\Omega u_i^{\rm ini}(x) {\rm d} x.
\end{equation}
The model is complemented with no-flux boundary conditions on the boundary $\partial\Omega$ of $\Omega$ for all species
\begin{align}
	(u_0\nabl u_i-u_i\nabl u_0+u_0u_i\beta z_i\nabl \phi)\cdot {\bm n} = 0,\ i=1,\ldots,n,\quad\text{ on }(0,T)\times\partial\Omega,
\end{align}
where ${\bm n}$ is the exterior unit normal vector to $\partial \Omega$.
The electric potential is assumed to satisfy the following equation
\begin{equation}\label{eq:phicont}
	-\lambda^2\Delta \phi= \sum_{i=1}^n z_i u_{i} + f,\quad\text{ in }(0,T)\times\Omega,
\end{equation}
under non-homogeneous Dirichlet boundary conditions on the part $\Gamma_D$ of the boundary, whereas homogeneous Neumann  boundary conditions are assumed on the complementary of $\Gamma_D$ in $\partial\Omega$:
\begin{equation}\label{eq:phidircont}
	\phi(t,x) = \phi^D(x)\hbox{ on } (0,T)\times\Gamma_D,\quad\quad\nabl \phi\cdot {\bm n} = 0\hbox{ on } (0,T)\times(\partial\Omega\setminus\Gamma_D).
\end{equation}
The total mass, the  electric charge and the specific diffusion coefficient of each species $i$  are respectively denoted by $M_i$, $z_i$ and $D_i$. We denote by $\beta$ the inverse thermal voltage, and by $\lambda$ the electrical permittivity.  The function $f$ models the background charge.
\label{eq:contmod}
\end{subequations}
\medskip
\\
The assumptions for System \eqref{eq:contmod}, used throughout the whole paper, are the following.

\begin{subequations}
\begin{align}
\bullet~ & \Omega \subset \mathbb{R}^d \mbox{ is polytopal  (polygonal if $d=2$, polyhedral if $d\ge 3$),} \nonumber \\
& \phantom{\Omega \subset \mathbb{R}^d\qquad} \mbox{ bounded, connected and open,} \label{hyp:omega} \\
\bullet~ & \partial\Omega = \Gamma_D\cup\Gamma_N  \mbox{ is a regular boundary, regularly partitioned into $\Gamma_D, \Gamma_N$,} \label{hyp:boundary} \\
\bullet~ & T>0 \mbox{ is the duration,} \label{hyp:T} \\
\bullet~ & \lambda,\beta \mbox{ are strictly positive parameters,} \label{hyp:lambdabeta} \\
\bullet~ & \mbox{For each species $i=1,\ldots,n$, assume that }z_i\in\mathbb{R}\mbox{ and } D_i>0, \label{hyp:species} \\
\bullet~ & u_i^{\rm ini}\in L^\infty(\Omega),\ i=0,\ldots,n \mbox{ are the initial concentrations with } \nonumber \\
& \phantom{u_i^{\rm ini}\in L^\infty(\Omega),\ i=0,\ldots,n } u_i^{\rm ini}\ge 0 \hbox{ and }\sum_{i=0}^n u_{i}^{\rm ini} = 1\mbox{ a.e. in }\Omega,\label{hyp:ini} \\
\bullet~ & M_i>0,\ i=0,\ldots,n \mbox{ are the total masses of solvent and species, } \label{hyp:masini} \\
\bullet~  & \phi^D \in H^1(\Omega),\mbox{ is a lifting for the Dirichlet boundary condition on $\phi$, } \label{hyp:phiD}\\
\bullet~  & f \in L^2(\Omega),\mbox{ models electric source terms. } \label{hyp:f}
\end{align}
\label{hyp:cont}
\end{subequations}
\\
In \cite{cances2019}, an approximation of System \eqref{eq:contmod} by a finite volume scheme is proven to converge under the assumption that all $D_i$ are equal and $\phi$ is identically equal to zero. In \cite{cances:hal-04777272}, a finite volume scheme, based on the Scharfetter-Gummel technique (see \cite{chainais2003fv} for the presentation of the technique), is proven to converge for any set of values $D_i>0$ in presence of an electric field.
There are also finite volume schemes proven to converge for other cross-diffusion systems, e.g. for the Maxwell-Stefan cross-diffusion system \cite{cances2024}, for the SKT system \cite{jungel2021} and for a nonlocal version \cite{herda2023} and for a larger class of cross-diffusion systems with volume-filling constraints \cite{jungel2023}.
Note that the mathematical convergence proofs done in \cite{cances2019,cances:hal-04777272, jungel2023,cances2024,jungel2021,herda2023} rely on the monotonicity properties resulting from the finite volume scheme with two-point flux approximation.\\
Again, using a finite volume scheme, an a posteriori error result is provided in \cite{berrens2025posteriorierrorcontrolfinite}. The limitation of two-point flux approximation schemes is the fact that they require admissible meshes satisfying an orthogonality condition, which are rarely available for general 3D domains. 
~\\
In \cite{gerstenmayer2019}, a finite element scheme is introduced to approximate System \eqref{eq:contmod}.
This scheme requires a regularization parameter $\varepsilon>0$ and that the initial conditions are strictly positive everywhere.
Under these assumptions and that the initial conditions are in $H^2(\Omega)$ the convergence is proven.
In \cite{braukhoff2022} a finite element scheme for a class of cross-diffusion systems is proven to converge requiring a regularization parameter $\varepsilon>0$ and a non-degeneracy condition on the diffusion matrix.
\medskip
\\
Our aim is to provide a control-volume finite element (CVFE) scheme (first presented in \cite{baliga1980} for convection diffusion equations), which applies on any simplicial mesh, and which enables the proofs of the existence of a discrete solution and of the convergence to a weak solution. The presented scheme also does not need a regularization parameter nor strictly positive initial conditions to converge.
The main idea is to follow the technique first presented in the seminal paper \cite{cancesguichard2016cvfe} for satisfying positivity constraints with nonlinear formulations, using a CVFE scheme.
\medskip
\\
Hence, we first present a CVFE scheme in Section \ref{sec:scheme}. Such a scheme cannot be  obtained by a direct Galerkin formulation of System \eqref{eq:contmod}. Indeed, using $u_0 =1 - \sum_{j=1}^n u_j$, System \eqref{eq:contmod} is exhibiting the following matrix form:
\begin{align*}
	\partial_t u_i - \sum_{m=1}^d\frac{\partial}{\partial x_m}\left(\sum_{j=1}^n A(u)_{i,j}\frac{\partial u_j}{\partial x_m}+ \left(1 - \sum_{j=1}^n u_j\right) u_i\beta z_i\frac{\partial \phi}{\partial x_m}\right)=0,
	\end{align*}
	where the diffusion matrix $A(u):[0,T]\times\Omega\to \mathbb{R}^{n\times n}$ is defined by
\begin{align*}
	A(u)_{i,j}:=
	\begin{cases}
		D_iu_i&\text{for } j\neq i\\
		\displaystyle D_i\left(1 - \sum_{k=1,k\neq i}^n u_k\right)&\text{for } j=i.
	\end{cases}
\end{align*}
The diffusion matrix $A(u)$ is then in general not symmetric nor positive semi-definite, and these missing properties make it usually difficult to obtain the existence of solutions and prove the convergence of simple Galerkin methods. To overcome these difficulties one defines the so-called entropy variables (see for example \cite{burger2012,cances:hal-04777272,jungel2015,gerstenmayer2019, gerstenmayer2018}) by
\begin{align*}
	\mu_i := \log\left(\frac{u_i}{u_0}\right)\quad\forall i=1,\dots,n.
\end{align*}
Notice that, contrary to what is done in \cite{cances:hal-04777272,gerstenmayer2019}, we do not include the electric potential into the entropy variable. Using these entropy variables the system can be rewritten as
\begin{equation}\label{eq:contchempot}\partial_t u_i - \sum_{m=1}^d\frac{\partial}{\partial x_m}\left(\sum_{j=1}^n B(u)_{i,j}\left(\frac{\partial \mu_j}{\partial x_m}+ \beta z_i\frac{\partial \phi}{\partial x_m}\right)\right)=0,
\end{equation}
with
\begin{align*}
	B(u)_{i,j} =
	\begin{cases}
		u_0u_i&\text{if } i=j\\
		0&\text{otherwise}.
	\end{cases}
\end{align*}
We then notice that the matrix $B(u)$ is diagonal and positive semi-definite, since it degenerates for $u_0=0$.
The choice of the unknowns $(\mu_i)$ for writing the numerical scheme is then based on this observation.  Several advantages are drawn from the use of entropy variables: the values $u_i$ remain strictly positive, and we obtain the proof of existence of at least one solution to the scheme. The lemmas used for this proof are given in the Appendix \ref{sec:existcheme} and are inspired by the existence proof done in  \cite{cancesguichard2016cvfe} (note that a specific difficulty arises from the fact that the diffusion of each ionic species degenerates if $u_0$ vanishes; we first prove that the discrete solvent concentration is bounded from below, before proving that all the discrete concentrations are bounded from below; this difficulty has not to be handled in \cite{gerstenmayer2019} due to the addition in the equations of a regularization term). Then uniform estimates of the numerical solution are derived in  Section \ref{sec:anascheme}. 
Using these estimates, we establish in Section \ref{sec:convscheme} the convergence of the discrete solutions, up to a subsequence, to a weak solution of the problem in the sense provided in \cite{gerstenmayer2018, jungel2015}, which is given as follows.

\begin{defn}[Weak Solution]\label{def:weaksolcont}
Under assumptions \ref{hyp:cont}, we denote by $H^1_D(\Omega)\subset H^1(\Omega)$  the set of all elements whose trace on $\Gamma_D$ is equal to zero.
We say that $(u_0,\dots,u_n,\phi)$ is a weak solution to System \eqref{eq:contmod} if the functions $u_0,\dots,u_n:(0,T)\times\Omega\to[0,1]$ and $\phi:(0,T)\times\Omega\to\mathbb{R}$ are measurable and such that
	\begin{subequations}
		\begin{equation}\label{eq:contspacesi}
			\sqrt{u_0}\in L^2(0,T;H^1(\Omega)),\quad \sqrt{u_0} u_i\in L^2(0,T;H^1(\Omega)),\quad i=1,\ldots,n,\quad\phi\in  L^2(0,T;H^1_D(\Omega)+\phi^D),
		\end{equation}
		\begin{multline}\label{eq:weakconti}
		\forall \psi\in C^\infty_c([0,T)\times \mathbb{R}^d),\\
			-\int_0^T \int_\Omega u_i \partial_t\psi {\rm d}x{\rm d}t-\int_\Omega u_i^{\rm ini} \psi(0,x){\rm d}x\\
			+  D_i \int_0^T \int_\Omega  \big(\sqrt{u_0}\nabl (\sqrt{u_0} u_i) - 3 \sqrt{u_0}u_i\nabl (\sqrt{u_0}) + u_0 u_i \beta z_i \nabl\phi\big)\cdot  \nabl\psi {\rm d}x{\rm d}t = 0,
		\end{multline}
		\begin{multline}\label{eq:weakcontphi}
		\forall v\in H^1_D(\Omega),\ \forall \xi \in C^\infty_c((0,T)),\\
			\lambda^2 \int_0^T \xi(t) \int_\Omega \nabl\phi(t,x)\cdot  \nabl v(x) {\rm d}x{\rm d}t = \int_0^T \xi(t)\int_\Omega  \left(\sum_{i=1}^n z_i u_{i}(t,x) + f(x)\right)v(x){\rm d}x{\rm d}t,
		\end{multline}
		\begin{align}
			\sum_{i=0}^n u_i = 1\quad\text{a.e. on } [0,T]\times\Omega.
		\end{align}
		\label{eq:weaksense}
	\end{subequations}
\end{defn}
We observe that the preceding weak sense for a solution to System \eqref{eq:contmod} is not straightforward, although one can formally check that $u_0\nabl u_i-u_i\nabl u_0 = \sqrt{u_0}\nabl (\sqrt{u_0} u_i) - 3 \sqrt{u_0}u_i\nabl (\sqrt{u_0})$. This is due to the fact that this sense must apply to the degenerate situations where $u_0 = 0$ in a part of the domain $(0,T)\times\Omega$, and that the functions involved in this weak sense must belong to function spaces such that all the functions which are integrated with respect to the time and the space variables are integrable.
\medskip\\
In Section \ref{sec:num}, we provide 2D and 3D numerical examples showing that the CVFE scheme behaves well, even in cases where $u_0$ tends to $0$ in a part of the domain $(0,T)\times\Omega$.
\section{The CVFE scheme}\label{sec:scheme}

Define a simplicial mesh $\mathcal{T}$ of $\Omega$ as a finite collection of open simplices (triangles if $d=2$ and tetrahedra if $d=3$), such that $\bigcup_{S\in\mathcal{T}}\overline{S}=\overline{\Omega}$ and $S\cap S'=\emptyset$ if $S\neq S'$. Standard compatibility conditions are assumed (no hanging nodes) in the sense of \cite{ciarlet}. The CVFE scheme relies on the definition of a dual mesh $\mathcal{M}$, associated to the vertices of the simplicial mesh. Hence, we denote by $\{x_K,\ K\in\mathcal{M}\}$ the set of the vertices of the discretisation $\mathcal{T}$.
We denote for every simplex $S\in\mathcal{T}$ by $\mathcal{M}_S$ the set of indices of the vertices of $S$. We also denote by $\mathcal{M}_D$ the set of the indices of the vertices located on $\Gamma_D$.
For every simplex $S\in\mathcal{T}$, we denote by $h_S$ the diameter of $S$ and by $\rho_S$ the maximal radius of an inscribed ball.
We then define the mesh size $|\mathcal{T}|$ and the mesh regularity $\theta_{\!\mathcal{T}}$ by
\begin{align*}
	|\mathcal{T}|=\max_{S\in\mathcal{T}} h_S,\quad\theta_{\!\mathcal{T}}=\max_{S\in\mathcal{T}} \frac{h_S}{\rho_S}.
\end{align*}
We define the dual barycentric mesh $\mathcal{M}$ as follows. For each $K\in\mathcal{M}$, we denote by $e_K$ the conforming $\mathbb{P}_1$ finite element basis function: this function is piecewise affine and continuous on $\overline{\Omega}$, it is equal to $1$ at the vertex $x_K$ and equal to $0$ at all the other vertices of the simplicial mesh. Then we denote as well by $K$ the set of all $x\in \Omega$ such that $e_K(x) >e_L(x)$ for any vertex  $L\neq K$ and we denote by $|K|$ its measure (area in 2D, volume in 3D). 
Notice that $|K|=\int_Ke_K{\rm d}x$.
In 2D, this means that $K$ is defined by joining the middle of the edges having $x_K$ as a vertex to the isobarycenter of the triangles $S\in\mathcal{T}$ such that $K\in\mathcal{M}_S$ (see Figure \ref{fig:dualmesh}).\\

\begin{figure}[!ht]
\begin{center}
\definecolor{qqqqff}{rgb}{0.,0.,1.}
\definecolor{vert}{rgb}{0.,.6,0.}
\definecolor{rouge}{rgb}{1.,0.,0.}
\definecolor{uuuuuu}{rgb}{0.26666666666666666,0.26666666666666666,0.26666666666666666}
\definecolor{rrrr}{rgb}{0,1,0.5}
\begin{tikzpicture}[scale=0.7]
\coordinate (A) at (1.06,-0.42);
\coordinate (B) at (2.,-5.);
\coordinate (C) at (9.43639634933273,-4.743885312873751);
\coordinate (D) at (11.93279269866546,0.09222937425249916);
\coordinate (E) at (12.,4.672229374252499);
\coordinate (F) at (3.5563963493327333,4.416114687126253);
\coordinate (G) at (6.2,2.16);
\coordinate (H) at (4.32,-0.52);
\coordinate (I) at (9.,0.);
\coordinate ({N_T}) at (6.8,-2.68);
\clip(0.5,-5.6) rectangle (12.5,5.5);
\draw[-, black!20, pattern=north east lines, pattern color=black!30, opacity=0.4] (I) -- (G) --  (H) -- cycle;
\coordinate (DE) at (barycentric cs:D=1,E=1);
\coordinate (DEI) at (barycentric cs:D=1,E=1,I=1);
\coordinate (EI) at (barycentric cs:I=1,E=1);
\coordinate (DI) at (barycentric cs:I=1,D=1);
\coordinate (CDI) at (barycentric cs:C=1,D=1,I=1);
\coordinate (CD) at (barycentric cs:C=1,D=1);
\coordinate (CI) at (barycentric cs:C=1,I=1);

\coordinate (EGI) at (barycentric cs:E=1,G=1,I=1);

\coordinate (CI) at (barycentric cs:C=1,I=1);
\coordinate (GI) at (barycentric cs:G=1,I=1);
\coordinate (EG) at (barycentric cs:G=1,E=1);
\coordinate (EFG) at (barycentric cs:G=1,F=1,E=1);
\coordinate (EF) at (barycentric cs:F=1,E=1);
\coordinate (FG) at (barycentric cs:F=1,G=1);
\coordinate (CIJ) at (barycentric cs:C=1,I=1,{N_T}=1);
\coordinate (IJ) at (barycentric cs:I=1,{N_T}=1);
\coordinate (HIJ) at (barycentric cs:H=1,I=1,{N_T}=1);
\coordinate (HI) at (barycentric cs:H=1,I=1);
\coordinate (GHI) at (barycentric cs:H=1,I=1,G=1);
\coordinate (GH) at (barycentric cs:H=1,G=1);
\coordinate (AGH) at (barycentric cs:H=1,A=1,G=1);
\coordinate (AG) at (barycentric cs:A=1,G=1);
\coordinate (AFG) at (barycentric cs:F=1,A=1,G=1);
\coordinate (AH) at (barycentric cs:A=1,H=1);
\coordinate (ABH) at (barycentric cs:B=1,A=1,H=1);
\coordinate (AB) at (barycentric cs:A=1,B=1);
\coordinate (BH) at (barycentric cs:H=1,B=1);
\coordinate (BHJ) at (barycentric cs:B=1,{N_T}=1,H=1);
\coordinate (BJ) at (barycentric cs:{N_T}=1,B=1);
\coordinate (BCJ) at (barycentric cs:B=1,{N_T}=1,C=1);
\coordinate (BC) at (barycentric cs:C=1,B=1);
\coordinate (CJ) at (barycentric cs:C=1,{N_T}=1);
\coordinate (HJ) at (barycentric cs:H=1,{N_T}=1);
\coordinate (AF) at (barycentric cs:A=1,F=1);
\fill[ red!30, opacity=0.3]  (BCJ) --(CJ)-- (CIJ) --(IJ)-- (HIJ)  --(HJ)-- (BHJ) --(BJ)-- cycle;
\draw [line width=1.6pt] (A)-- (B);
\draw [line width=1.6pt] (B)-- (C);
\draw [line width=1.6pt] (C)-- (D);
\draw [line width=1.6pt] (D)-- (E);
\draw [line width=1.6pt] (F)-- (A);
\draw [line width=1.2pt] (E)-- (G);
\draw [line width=1.2pt] (A)-- (H);
\draw [line width=1.2pt] (H)-- (B);
\draw [line width=1.2pt] (G)-- (A);
\draw [line width=1.2pt] (G)-- (H);
\draw [line width=1.2pt] (G)-- (F);
\draw [line width=1.2pt] (I)-- (D);
\draw [line width=1.2pt] (I)-- (C);
\draw [line width=1.2pt] (I)-- (G);
\draw [line width=1.2pt] (E)-- (I);
\draw [line width=1.2pt] (H)-- (I);
\draw [line width=1.6pt] (F)-- (E);
\draw [line width=1.2pt] (I)-- ({N_T});
\draw [line width=1.2pt] (C)-- ({N_T});
\draw [line width=1.2pt] ({N_T})-- (B);
\draw [line width=1.2pt] ({N_T})-- (H);
\draw [line width=2.pt,dotted,color=qqqqff] (DE)-- (DEI);
\draw [line width=2.pt,dotted,color=qqqqff] (EI)-- (DEI);
\draw [line width=2.pt,dotted,color=qqqqff] (DEI)-- (DI);
\draw [line width=2.pt,dotted,color=qqqqff] (DI)-- (CDI);
\draw [line width=2.pt,dotted,color=qqqqff] (CDI)-- (CD);
\draw [line width=2.pt,dotted,color=qqqqff] (CDI)-- (CI);
\draw [line width=2.pt,dotted,color=qqqqff] (EGI)-- (EI);
\draw [line width=2.pt,dotted,color=qqqqff] (EGI)-- (GI);
\draw [line width=2.pt,dotted,color=qqqqff] (EGI)-- (EG);
\draw [line width=2.pt,dotted,color=qqqqff] (EFG)-- (EG);
\draw [line width=2.pt,dotted,color=qqqqff] (EFG)-- (EF);
\draw [line width=2.pt,dotted,color=qqqqff] (EFG)-- (FG);
\draw [line width=2.pt,dotted,color=qqqqff] (CIJ)-- (CI);
\draw [line width=2.pt,dotted,color=qqqqff] (CIJ)-- (IJ);
\draw [line width=2.pt,dotted,color=qqqqff] (HIJ)-- (IJ);
\draw [line width=2.pt,dotted,color=qqqqff] (HIJ)-- (HI);
\draw [line width=2.pt,dotted,color=qqqqff] (GHI)-- (HI);
\draw [line width=2.pt,dotted,color=qqqqff] (GHI)-- (GI);
\draw [line width=2.pt,dotted,color=qqqqff] (GHI)-- (GH);
\draw [line width=1.2pt,dotted,color=qqqqff] (AGH)-- (GH);
\draw [line width=2.pt,dotted,color=qqqqff] (GH)-- (AGH);
\draw [line width=2.pt,dotted,color=qqqqff] (AGH)-- (AG);
\draw [line width=2.pt,dotted,color=qqqqff] (AG)-- (AFG);
\draw [line width=2.pt,dotted,color=qqqqff] (AFG)-- (FG);
\draw [line width=2.pt,dotted,color=qqqqff] (AGH)-- (AH);
\draw [line width=2.pt,dotted,color=qqqqff] (ABH)-- (AH);
\draw [line width=2.pt,dotted,color=qqqqff] (ABH)-- (AB);
\draw [line width=2.pt,dotted,color=qqqqff] (ABH)-- (BH);
\draw [line width=2.pt,dotted,color=qqqqff] (BH)-- (BHJ);
\draw [line width=2.pt,dotted,color=qqqqff] (BHJ)-- (BJ);
\draw [line width=2.pt,dotted,color=qqqqff] (BJ)-- (BCJ);
\draw [line width=2.pt,dotted,color=qqqqff] (BCJ)-- (BC);
\draw [line width=2.pt,dotted,color=qqqqff] (BCJ)-- (CJ);
\draw [line width=2.pt,dotted,color=qqqqff] (CJ)-- (CIJ);
\draw [line width=2.pt,dotted,color=qqqqff] (BHJ)-- (HJ);
\draw [line width=2.pt,dotted,color=qqqqff] (HJ)-- (HIJ);
\draw [line width=2.pt,dotted,color=qqqqff] (AF)-- (AFG);
\draw [color=rouge](I) node[anchor=north west] {$x_L$};
\draw [color=rouge]({N_T}) node[anchor=north] {$x_K$};
\draw (GHI) node[anchor=north west] {$S$};
\begin{scriptsize}
\draw [fill=black] (A) circle (2.0pt);
\draw [fill=black] (B) circle (2.0pt);
\draw [fill=uuuuuu] (C) circle (2.0pt);
\draw [fill=uuuuuu] (D) circle (2.0pt);
\draw [fill=uuuuuu] (E) circle (2.0pt);
\draw [fill=uuuuuu] (F) circle (2.0pt);
\draw [fill=black] (G) circle (2.0pt);
\draw [fill=rouge,color=rouge] ({N_T}) circle (4.0pt);
\draw [fill=rouge,color=rouge] (I) circle (4.0pt);
\draw [color=qqqqff] (DEI) circle (2.0pt);
\draw [color=qqqqff] (CDI) circle (2.0pt);
\draw [color=qqqqff] (CIJ) circle (2.0pt);
\draw [color=qqqqff] (HIJ) circle (2.0pt);
\draw [color=qqqqff] (BCJ) circle (2.0pt);
\draw [color=qqqqff] (BHJ) circle (2.0pt);{2r}
\draw [color=qqqqff] (ABH) circle (2.0pt);
\draw [color=qqqqff] (AFG) circle (2.0pt);
\draw [color=qqqqff] (AGH) circle (2.0pt);
\draw [color=qqqqff] (GHI) circle (2.0pt);
\draw [color=qqqqff] (EFG) circle (2.0pt);
\draw [color=qqqqff] (EGI) circle (2.0pt);
\end{scriptsize}
\end{tikzpicture}
\caption{Triangle $S\in \mathcal{T}$ (solid line) and dual cell $K\in \mathcal{M}$ (dashed line).}
\label{fig:dualmesh}
\end{center}
\end{figure}
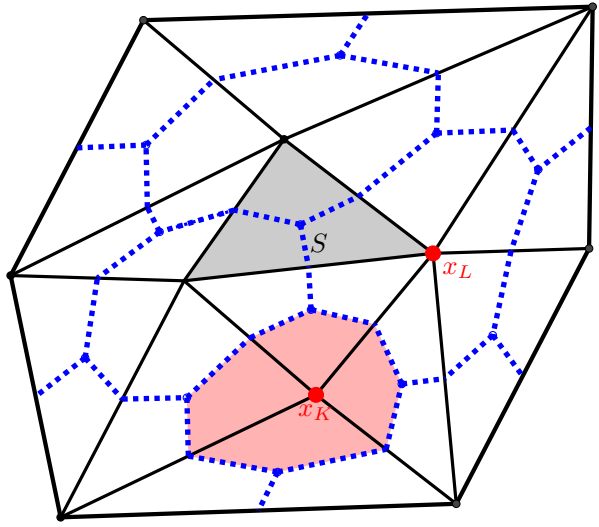

We construct two discrete spaces.
The first one is the classical conforming $\mathbb{P}_1$-finite element space corresponding to the simplicial mesh $\mathcal{T}$, i.e.
\begin{align*}
	V_{\!\mathcal{T}}:=\{f\in H^1(\Omega)\,|\,f|_S\in \mathbb{P}_1\quad\forall S\in\mathcal{T}\},
\end{align*}
 and we denote by
\begin{align*}
	V_{\!\mathcal{T},D}:=\{v\in V_{\!\mathcal{T}}\,|\,v(x_K) = 0\quad\forall K\in\mathcal{M}_D\}.
\end{align*}
We assume, throughout the paper, the existence of a Poincaré constant $C_P>0$ such that
\begin{equation}\label{eq:poincare}
 \forall v\in H^1_D(\Omega), \Vert v\Vert_{L^2(\Omega)}\le C_P \Vert \nabl v\Vert_{L^2(\Omega)},
\end{equation}
and we assume that the Dirichlet boundary and the mesh are such that $V_{\!\mathcal{T},D}\subset H^1_D(\Omega)$.
For any $f\in V_{\!\mathcal{T}}$ and $S\in \mathcal{T}$, we denote by $\nabl f(S)$ the constant value of $\nabl f$ in $S$.
A basis of $V_{\!\mathcal{T}}$ is given by $(e_K)_{K\in\mathcal{M}}$. 

We also assume the existence of $\phi^D_{\!\mathcal{T}}\in V_{\!\mathcal{T}}$ (which is meant to converge to  $\phi^D$ in $H^1(\Omega)$) and of a constant $C_\phi^D>0$, independent of the discretisation, such that
\begin{equation}\label{eq:cphiD}
 \Vert \phi^D_{\!\mathcal{T}}\Vert_{H^1(\Omega)}\le C_\phi^D.
\end{equation}
The second discrete space $X_{\!\mathcal{M}}$ consists of piecewise constant functions on the dual mesh $\mathcal{M}$.
For any sequence $v := (v_K)_{K\in\mathcal{M}}\in \mathbb{R}^{\mathcal{M}}$, we define the functions $v_{\!\mathcal{T}}\in V_{\!\mathcal{T}}$ and $v_{\!\mathcal{M}}\in X_{\!\mathcal{M}}$ with
\begin{align*}
v_{\!\mathcal{T}}(x_K)=v_K  \hbox{ and }v_{\!\mathcal{M}}(x)=v_K \hbox{ for a.e. }x\in K,\quad\forall K\in\mathcal{M}.
\end{align*}
We also define the following discrete $H^{-1}$-norm on $X_{\!\mathcal{M}}$: for any sequence $v = (v_K)_{K\in\mathcal{M}}\in \mathbb{R}^{\mathcal{M}}$, we set
\begin{align*}
	\|v_{\!\mathcal{M}}\|_{H^{-1}_{\!\mathcal{M}}(\Omega)}=\sup_{w\in \mathbb{R}^{\mathcal{M}}\setminus\{0\}}\frac{\sum_{K\in\mathcal{M}}|K|v_Kw_K}{\|w_{\!\mathcal{T}}\|_{H^1(\Omega)}}.
\end{align*}
We now define a time sequence $\tau = (t^{k})_{k=0,\ldots,N_T}$ with $t^0 = 0 < t^1\ldots < t^{N_T} = T$ and $N_t\in\mathbb{N}$. We then denote by $\tau^{k} = t^{k} - t^{k-1}$ for $k=1,\ldots,{N_T}$ and we denote by $|\tau|\in\mathbb{R}$ the maximum value of the time step:
\begin{equation}\label{eq:deftau}
 |\tau| = \max\{\tau^{k},\ k=1,\ldots,N_T\}.
\end{equation}
We define the time-space dependent spaces
\begin{align*}
	V_{\!\mathcal{T},\tau}&=\{u:(0,T]\times \Omega\to\mathbb{R}\,|\, u(t,\cdot)=u(t^{k+1},\cdot)\in V_{\!\mathcal{T}},\,\forall t\in(t^k,t^{k+1}],\ k=0,\ldots,{N_T}-1\},\\
	X_{\!\mathcal{M},\tau}&=\{u:(0,T]\times \Omega\to\mathbb{R}\,|\, u(t,\cdot)=u(t^{k+1},\cdot)\in X_{\!\mathcal{M}},\,\forall t\in(t^k,t^{k+1}],\ k=0,\ldots,{N_T}-1\}.
\end{align*}
\medskip\\
The primary unknowns of the scheme are the real values $\mu_{i,K}^{k}$ for $i=1,\ldots,n$ and $\phi_{K}^{k}$, for $K\in \mathcal{M}$ and $k=1,\ldots,{N_T}$.  The scheme is defined by the following equations.
\begin{subequations}
The initial condition is taken into account by
\begin{equation}\label{eq:uiini}
 u_{i,K}^{0} = \frac 1 {|K|} \int_K u_{i}^{\rm ini}(x) {\rm d} x, \ i=0,\ldots,n,\ K\in \mathcal{M}.
\end{equation}
We set, for any $i=0,\ldots,n$, $K\in \mathcal{M}$ and $k=1,\ldots,{N_T}$,
\begin{equation}\label{eq:uimui}
 u_{i,K}^{k} = u_{0,K}^{k}\exp(\mu_{i,K}^{k}),
\end{equation}
with
\begin{equation}\label{eq:uz}
 u_{0,K}^{k} = 1 - \sum_{i=1}^n u_{i,K}^{k},
\end{equation}
which imposes
\[
 u_{0,K}^{k} = \frac 1 {1 + \sum_{i=1}^n\exp(\mu_{i,K}^{k})} \in (0,1) \hbox { and therefore }u_{i,K}^{k}\in (0,1).
\]
Using the finite volume side of the scheme and following the form \eqref{eq:contchempot} of the continuous equations, the approximation of the conservation of the species $i$ is given by
\begin{equation}\label{eq:ui}
	|K| \frac {u_{i,K}^{k} - u_{i,K}^{k-1}}{\tau^{k}} + \sum_{S\in \mathcal{T}_K}\sum_{L\in \mathcal{M}_{S}} {u}_{0,S}^{k} {u}_{i,S}^{k} F_{i,K,L}^{k,S} = 0
\end{equation}
where, using the finite element side of the scheme, the flux $F_{i,K,L}^{k,S}$ is given by
\begin{equation}\label{eq:fi}
	F_{i,K,L}^{k,S} = D_i a_{KL}^S((\mu_{i,K}^{k}-\mu_{i,L}^{k})+\beta z_i(\phi_K^{k}-\phi_L^{k}))
\end{equation}
with  
\[
	a_{KL}^S = -\int_S \nabl e_K\cdot\nabl e_L\,{\rm d}x = - |S| \nabl e_K(S)\cdot\nabl e_L(S).
\]
The values ${u}_{0,S}^{k}$ and ${u}_{i,S}^{k}$ are either defined by
\begin{equation}\label{eq:uztmax}
	{u}_{0,S}^{k} = \max \{u_{0,L}^{k}, L\in \mathcal{M}_S\}, ~{u}_{i,S}^{k} = \max \{u_{i,L}^{k} , L\in \mathcal{M}_S\}
\end{equation}
or by
\begin{equation}\label{eq:uztave}
	{u}_{0,S}^{k} = \frac 1 {d+1} \sum_{L\in \mathcal{M}_S}u_{0,L}^{k}, ~{u}_{i,S}^{k} = \frac 1 {d+1} \sum_{L\in \mathcal{M}_S}u_{i,L}^{k}.
\end{equation}
In both cases the following inequalities hold
\begin{equation}\label{eq:uzt}
	u_{0,L}^{k} \le (d+1){u}_{0,S}^{k}\hbox{ and }u_{i,L}^{k}\le (d+1){u}_{i,S}^{k} \hbox{ for all } L\in \mathcal{M}_S.
\end{equation}
Finally, using the discrete lifting $\phi^D_{\!\mathcal{T}}$, we prescribe the non-homogeneous Dirichlet boundary condition for the electric potential by setting
\begin{equation}\label{eq:phiD}
	\phi^{k}_K = (\phi^D_{\!\mathcal{T}})_K, \ \forall K\in \mathcal{M}_D,
\end{equation}
and we consider the following scheme for the electric field:
\begin{equation}
 \label{eq:phidisc}
	\lambda^2\sum_{S\in \mathcal{T}_K}\sum_{L\in \mathcal{M}_S}a_{KL}^S (\phi^{k}_{K}-\phi^{k}_{L})=|K|\sum_{i=1}^n z_i u_{i,K}^{k} + \int_K f(x){\rm d}x,\ \forall K\in \mathcal{M}\setminus\mathcal{M}_D.
\end{equation}
Only for the purpose of computing an estimate on the discrete time derivative of $\phi$, we also define $\phi^0_{\!\mathcal{T}} \in V_{\!\mathcal{T}}$ by \eqref{eq:phiD} and \eqref{eq:phidisc} with $k=0$.
Note that the computation of $\phi^0_{\!\mathcal{T}}$ is not needed for the practical implementation of the solution.
\label{eq:scheme}\end{subequations}
In the following, unless it is necessary, the scheme will be referred to as Scheme \eqref{eq:scheme} without specifying the choice done between \eqref{eq:uztmax} and \eqref{eq:uztave}, using that both are verifying \eqref{eq:uzt}.
We prove in this paper that the scheme is convergent, independently of this choice.
However the numerical convergence orders differ with the choice \eqref{eq:uztmax} and \eqref{eq:uztave}.

Let us state the existence theorem of at least one solution to the discrete scheme, the proof of which is provided in Appendix \ref{sec:existcheme}.
\begin{thm}\label{theo:existsolscheme}
	There exists at least one solution $((\mu_{i,K}^{k})_{i,K},(\phi_{K}^{k})_K)_{k=1,\dots,{N_T}}$ to Scheme \eqref{eq:scheme}.
\end{thm}
Any solution to Scheme \eqref{eq:scheme} $((\mu_{1,K}^k)_{K\in\mathcal{M}})_{k=1,\dots,N_T},\dots,((\mu_{n,K}^k)_{K\in\mathcal{M}})_{k=1,\dots,N_T},((\phi_{K}^k)_{K\in\mathcal{M}})_{k=1,\dots,N_T}$ is a collection of elements of $(\mathbb{R}^{\mathcal{M}})^{N_T}$ which can be represented in $V_{\mathcal{T},\tau}$ and $X_{\mathcal{M},\tau}$.
To distinguish between these two representations we introduce the following notation.
\begin{defn}\label{def:appsol}
	Let $v=((v_K^k)_{K\in\mathcal{M}})_{k=1,\dots,N_T}\in(\mathbb{R}^{\mathcal{M}})^{N_T}$.
	Denote by $[v]_{\!\mathcal{T}}$ (resp. $[v]_{\!\mathcal{T}}^k$) the function in $V_{\mathcal{T},\tau}$ (resp. $V_{\mathcal{T}}$) such that
	\begin{align*}
		[v]_{\!\mathcal{T}}(t_{k},x_K):=[v]_{\!\mathcal{T}}^k(x_K) := v_K^k\quad\forall K\in\mathcal{M} \text{ and }k=1,\dots,N_T.
	\end{align*}
	Further denote by $[v]_{\!\mathcal{M}}$ (resp. $[v]_{\!\mathcal{M}}^k$) the function in $X_{\mathcal{M},\tau}$ (resp. $X_{\mathcal{M}}$) such that
	\begin{align*}
		[v]_{\!\mathcal{M}}(t_k,x):=[v]_{\!\mathcal{M}}^k(x):=v_K\hbox{ for a.e. } x\in K,\ \forall K\in\mathcal{M}  \text{ and }k=1,\dots,N_T.
	\end{align*}
\end{defn}
For a solution $((\mu_{i,K}^{k})_{i,K},(\phi_{K}^{k})_K)_{k=1,\dots,{N_T}}$ to Scheme \eqref{eq:scheme}, for $i=0,\dots,n$ we obtain the functions $\lunk u_i\runk_{\!\mathcal{M}},\lunk u_i\runk_{\!\mathcal{T}}$ as the representatives of $((u_{i,K}^k)_{K\in\mathcal{M}}^k)$ and $\lunk \phi\runk_{\!\mathcal{M}},\lunk \phi\runk_{\!\mathcal{T}}$ as the representatives of $((\phi_K^k)_{K\in\mathcal{M}})_{k=1,\dots,N_T}$.
For $u_i,u_j\in(\mathbb{R}^{\mathcal{M}})^{N_T}$ and a function $f:\mathbb{R}\times\mathbb{R}\to\mathbb{R}$ we denote by $f(v,w) := ((f(v_K^k,w_K^k))_{K\in\mathcal{M}})_{k=1,\dots,N_T}\in(\mathbb{R}^{\mathcal{M}})^{N_T}$ the pointwise application of the function $f$ to $v$ and $w$.
A great advantage of the CVFE scheme and of the piecewise constant representation of the unknown functions, is that the following holds:
\[
 \lunk f(u_{i},u_{j})\runk_{\!\mathcal{M}}(t,x) = f(\lunk u_i\runk_{\!\mathcal{M}}(t,x),\lunk u_j\runk_{\!\mathcal{M}}(t,x))\hbox{ for a.e. }(t,x)\in (0,T)\times\Omega.
\]
This will be useful for the convergence properties of the scheme.
For $v=((v_K^k)_{K\in\mathcal{M}})_{k=0,\dots,N_T}\in(\mathbb{R}^{\mathcal{M}})^{N_T+1}$ define the discrete time derivate of $v$ by
\begin{align*}
	\overline{\partial}_t v := \left(\left(\frac{v_{K}^k-v_{K}^{k-1}}{\tau^k}\right)_{K\in\mathcal{M}}\right)_{k=1,\dots,N_T}.
\end{align*}

\section{Entropy inequality and uniform estimates}\label{sec:anascheme}

We consider a given discretisation $(\mathcal{T},\tau)$, following the previous section.
Let $\theta\ge \theta_{\!\mathcal{T}}$ be given. In the sequel, we denote by $C_i$ with  $i\in\mathbb{N}$ various positive constants which may depend on $d$, $T$, $\Omega$, $(z_i,M_i,D_i)_i$, $\lambda$, $\beta$, $\Vert f\Vert_{L^2(\Omega)}$, $\theta$, $ C_\phi^D$ and $C_P$ but not on $|\mathcal{T}|$ nor on $\tau$. The values $C_{\min}$ and $C_{\max}$ used in the following proofs, which only depend on $d$ and $\theta$, are defined in Lemma \ref{lemma:equivalenz_of_H1L2norms} which states the equivalence of some norms on $X_\mathcal{M}$ and $V_\mathcal{T}$.
\begin{lemma}\label{lem:electric}
There exists $\ctel{cte:phi}>0$  such that, for any solution $((\mu_{i,K}^{k})_{i,K},(\phi_{K}^{k})_K)_k$ of  \eqref{eq:scheme},
\begin{equation}\label{eqn:defC_phi}
\Vert \lunk \phi\runk^{k}_{\!\mathcal{T}}\Vert_{H^1(\Omega)} \le \cter{cte:phi}, \hbox{ for all }k=1,\ldots,{N_T}.
\end{equation}
\end{lemma}
\begin{proof}
We get from \eqref{eq:phidisc} that the following weak formulation holds:
\begin{equation}
 \label{eq:phi}
	\lambda^2\int_\Omega\nabl\lunk \phi\runk_{\!\mathcal{T}}^{k}(x)\cdot\nabl v_{\!\mathcal{T}}(x){\rm d}x=\int_\Omega \Big(\sum_{i=1}^n z_i \lunk u_i \runk_{\!\mathcal{M}}^{k}(x) + f(x)\Big)v_{\!\mathcal{M}}(x){\rm d}x, \ \forall v\in V_{\!\mathcal{T},D}.
\end{equation}
Letting  $v_{\!\mathcal{T}} = \lunk \phi\runk_{\!\mathcal{T}}^{k}-\phi^D_{\!\mathcal{T}}$ in \eqref{eq:phi} yields
\[
\lambda^2\int_\Omega |\nabl v_{\!\mathcal{T}}(x)|^2 {\rm d}x=\int_\Omega \Big(\sum_{i=1}^n z_i \lunk u_i \runk_{\!\mathcal{M}}^{k}(x) + f(x)\Big)v_{\!\mathcal{M}}(x){\rm d}x - \lambda^2\int_\Omega\nabl\phi^D_{\!\mathcal{T}}(x)\cdot\nabl v_{\!\mathcal{T}}(x){\rm d}x.
\]
Using $0\leq u_i^{k}\leq 1$ and \eqref{eq:poincare}, we get
\[
\lambda^2\Vert \nabl v_{\!\mathcal{T}}\Vert_{L^2(\Omega)^d} \le \left(|\Omega|^{1/2} \sum_{i=1}^n |z_i| + \Vert f\Vert_{L^2(\Omega)}\right)  C_{\max} C_P + \lambda^2\Vert \nabl \phi^D_{\!\mathcal{T}}\Vert_{L^2(\Omega)^d}.
\]
With $\Vert v_{\!\mathcal{T}}\Vert_{L^2(\Omega)} \le C_P\Vert \nabl v_{\!\mathcal{T}}\Vert_{L^2(\Omega)^d}$ follows for some $C_2>0$
\[
\Vert v_{\!\mathcal{T}}\Vert_{H^1(\Omega)} \le \ctel{cte:oneb}.
\]
Writing
\[
\Vert \lunk \phi\runk^{k}_{\!\mathcal{T}}\Vert_{H^1(\Omega)} \le \Vert v_{\!\mathcal{T}}\Vert_{H^1(\Omega)} + \Vert \phi^D_{\!\mathcal{T}}\Vert_{H^1(\Omega)},
\]
concludes the proof.
\end{proof}

Let us denote by 
\begin{equation}\label{eq:defzeta}
	\zeta:[0,+\infty)\to[0,+\infty), x\mapsto
	\begin{cases}
		0&\text{if } x=0\\
		x(\log{x}-1) +1&\text{otherwise}.
	\end{cases}
\end{equation}
Remark that, for $x\in [0,1]$, $\zeta(x)\in [0,1]$.
We now state the entropy inequality that allows us to obtain a priori bounds for any solution of Scheme \eqref{eq:scheme}.
\begin{lemma}\label{lem:entropy_estimate}
	For any solution $((\mu_{i,K}^{k})_{i,K},(\phi_{K}^{k})_K)_k$ of  \eqref{eq:scheme}, the following entropy inequalities hold for every $k=1,\dots,N_T$:
	\begin{multline}\label{eq:entropone}
		\sum_{k=1}^{N_T}\tau^{k}\frac {C_{\min} } 2\sum_{i=1}^n D_i\sum_{S\in\mathcal{T}} \frac{|S|}{h_S^2}
		{u}_{0,S}^{k}{u}_{i,S}^{k}\sum_{ \{K,L\}\in\mathcal{E}_S }(\mu_{i,K}^{k} - \mu_{i,L}^{k})^2\\
		\leq \sum_{k=1}^{N_T}\frac {\tau^{k}} 2\sum_{i=1}^n\sum_{K\in\mathcal{M}}\sum_{S\in \mathcal{T}_K} |S| {u}_{0,S}^{k} {u}_{i,S}^{k} D_i|\nabl\lunk\mu_i\runk^{k}_{\!\mathcal{T}}(S)|^2\\
		\leq \sum_{i=1}^n \sum_{K\in\mathcal{M}}|K|\zeta(u_{i,K}^{0})+ T \cter{cte:phid},
	\end{multline}
	which yields the existence of $\ctel{cte:estiment}>0$ such that
	\begin{equation}\label{eq:entroptwo}
		\sum_{i=1}^n\left\|\nabl \lunk u_i\sqrt{u_0}\runk_{\!\mathcal{T}}\right\|_{L^2((0,T)\times\Omega)}^2+\left\|\nabl \lunk u_0 \runk_{\!\mathcal{T}}\right\|_{L^2((0,T)\times\Omega)}^2+\left\|\nabl\lunk\sqrt{u_0}\runk_{\!\mathcal{T}}\right\|_{L^2((0,T)\times\Omega)}^2
		\leq  \cter{cte:estiment}.
	\end{equation}
\end{lemma}
\begin{proof}

The multiplication of \eqref{eq:ui} by $\mu_{i,K}^{k}$ and the sum on $K\in\mathcal{M}$ and $i=1,\dots,n$ yields
\begin{align*}
	\underbrace{\sum_{i=1}^n\sum_{K\in\mathcal{M}}|K|\frac{u_{i,K}^{k}-u_{i,K}^{k-1}}{\tau^{k}}\mu_{i,K}^{k}}_{=:A}+\underbrace{\sum_{i=1}^n\sum_{K\in\mathcal{M}}\sum_{S\in\mathcal{T}_K}\sum_{L\in\mathcal{V_S}}F_{i,K,L}^{k,S}\mu_{i,K}^{k}}_{=:B}=0.
\end{align*}
For part $A$, using that $\sum_{i=1}^n u_{i,K} = 1-u_{0,K}$, we get
\begin{align*}
	A=\sum_{i=1}^n\sum_{K\in\mathcal{M}}|K|\frac{u_{i,K}^{k}-u_{i,K}^{k-1}}{\tau^{k}}\left(\log(u_{i,K}^{k})-\log(u_{0,K}^{k})\right)
	=
	\sum_{i=0}^n\sum_{K\in\mathcal{M}}|K|\frac{u_{i,K}^{k}-u_{i,K}^{k-1}}{\tau^{k}}\log(u_{i,K}^{k}).
\end{align*}
The convexity of the function $\zeta$ yields
\begin{align*}
	A\geq \sum_{i=0}^n|K|\frac{u_{i,K}^{k}(\log(u_{i,K}^{k})-1)-u_{i,K}^{k-1}(\log(u_{i,K}^{k-1})-1)}{\tau^{k}} = \sum_{i=0}^n|K|\frac{\zeta(u_{i,K}^{k})-\zeta(u_{i,K}^{k-1})}{\tau^{k}}.
\end{align*}
To estimate $B$ we split $B$ into
\begin{multline*}
	B = \sum_{i=1}^n\sum_{K\in\mathcal{M}}\sum_{S\in\mathcal{T}_K} |S|D_i\left(\nabl\lunk\mu_i\runk^{k}_{\!\mathcal{T}}(S)+\beta z_i\nabl\lunk \phi\runk_{\!\mathcal{T}}^{k}(S)\right)\cdot \nabl e_K(S)\mu_{i,K}^{k}\\
	=  \underbrace{\sum_{i=1}^n\sum_{K\in\mathcal{M}}\sum_{S\in \mathcal{T}_K} |S| {u}_{0,S}^{k} {u}_{i,S}^{k} D_i|\nabl\lunk\mu_i\runk^{k}_{\!\mathcal{T}}(S)|^2}_{=:B_1}
	+\underbrace{\sum_{i=1}^n\sum_{S\in \mathcal{T}}|S|D_i{u}_{0,S}^{k}{u}_{i,S}^{k} \beta z_i \nabl\lunk \phi\runk_{\!\mathcal{T}}^{k}(S) \cdot \nabl\lunk\mu_i\runk^{k}_{\!\mathcal{T}}(S)}_{=:B_2}.
\end{multline*}
Then, using the Young inequality, we get $B_2 \ge -\frac 1 2 (B_1 + B_3)$, with
\begin{align*}
	B_3 = \sum_{i=1}^n\sum_{S\in \mathcal{T}} |S| {u}_{0,S}^{k} {u}_{i,S}^{k} D_i (\beta z_i)^2 |\nabl\lunk \phi\runk_{\!\mathcal{T}}^{k}(S) |^2 \le  \sum_{i=1}^n\sum_{S\in \mathcal{T}} |S| D_i (\beta z_i)^2 |\nabl\lunk \phi\runk_{\!\mathcal{T}}^{k}(S) |^2.
\end{align*}
We remark that  $\frac 1 2B_3\le \ctel{cte:phid}$, using Lemma \ref{lem:electric}, which, combined with the estimate for $A$ and summing on $k=1,\ldots N_T$,  yields the right inequality in \eqref{eq:entropone}.
Using \eqref{eq:gradt} in Lemma \ref{lemma:equivalenz_of_H1L2norms}, we get
	\begin{align}\label{eqn:entopy_mixed_sqrt}
	B_1 \ge B_{11} := C_{\min} \sum_{i=1}^n\sum_{S\in \mathcal{T}} \frac {|S|} {h_S^2} {u}_{0,S}^{k} {u}_{i,S}^{k} D_i\sum_{ \{K,L\}\in\mathcal{E}_S } (\mu_{i,L}^{k}-\mu_{i,K}^{k})^2,
\end{align}
which provides the left inequality in \eqref{eq:entropone}.

\medskip

Let us turn to the proof of \eqref{eq:entroptwo}. Using \eqref{eq:(i)} in Lemma \ref{lemma:useful_inequalities} and property \eqref{eq:uzt} of ${u}_{i,S}^{k}$ and ${u}_{i,S}^{k}$, we can write
\[
 B_{11} \ge  \frac {C_{\min}}{(d+1)^2} \sum_{i=1}^n\sum_{S\in \mathcal{T}} \frac {|S|} {h_S^2}   D_i\sum_{ \{K,L\}\in\mathcal{E}_S } \left(\sqrt{u_{i,L}^{k} u_{0,K}^{k}}-\sqrt{u_{i,K}^{k} u_{0,L}^{k}}\right)^2.
\]
We now use \eqref{eq:(ii)} in Lemma \ref{lemma:useful_inequalities} to obtain
\begin{multline*}
	B_{11} \geq \ctel{cte:bone}\sum_{S\in\mathcal{T}}\frac{|S|}{h_S}\sum_{\{K,L\}\in\mathcal{E}_S} \Big(\sum_{i=1}^n\big(u_{i,L}^{k}\sqrt{u_{0,L}^{k}}-u_{i,K}^{k}\sqrt{u_{0,K}^{k}}\big)^2\\
	+ \big(u_{0,L}^{k} -u_{0,K}^{k}\big)^2
	+\big(\sqrt{u_{0,L}^{k}}-\sqrt{u_{0,K}^{k}}\big)^2\Big).
\end{multline*} 
Using \eqref{eq:gradt} in Lemma \ref{lemma:equivalenz_of_H1L2norms} finally yields
\begin{equation}
	B_{11} \ge \ctel{cte:btwo}\sum_{S\in \mathcal{T}} |S| \left(\sum_{i=1}^n\left|\nabl \lunk u_i\sqrt{u_0}\runk^{k}_{\!\mathcal{T}}(S)\right|^2+\left|\nabl \lunk u_0\runk^{k}_{\!\mathcal{T}}(S)\right|^2+\left|\nabl\lunk \sqrt{u_0}\runk^{k}_{\!\mathcal{T}}(S)\right|^2\right).
\end{equation}
This implies \eqref{eq:entroptwo}.
\end{proof}

The following bounds on the time derivatives are later used to obtain the compactness of the scheme in Section \ref{sec:convscheme}.
\begin{lemma}\label{lemma:time_derivative_estimate}
	There exists  $\ctel{cte:time}>0$ such that, for any solution $((\mu_{i,K}^{k})_{i,K},(\phi_{K}^{k})_K)_k$ of  \eqref{eq:scheme},  we have for all $i=0,\dots,n$
	\begin{align*}
		\|\lunk\overline{\partial}_t  u_i \runk_{\!\mathcal{M}}\|_{L^2(0,T;H^{-1}_{\!\mathcal{M}}(\Omega))}\leq \cter{cte:time}.
	\end{align*}
\end{lemma}
\begin{proof}
	Let $\varphi\in V_{\!\mathcal{T},\tau}$.
	We multiply equation \eqref{eq:ui} with $\tau^{k}\varphi_K^{k}$ and sum over all $K\in\mathcal{M}$ and $k=1,\dots,{N_T}$ to obtain
	\begin{align*}
		\sum_{k=1}^{N_T}\sum_{K\in\mathcal{M}}|K|\left(u_{i,K}^{k}-u_{i,K}^{k-1}\right)\varphi_K^{k}
		&=\sum_{k=1}^{N_T}\tau^{k}\sum_{K\in\mathcal{M}}\sum_{S\in\mathcal{T}_K}|S|{u}_{i,S}^{k}{u}_{0,S}^{k}D_i\left(\nabl\lunk\mu_i\runk^{k}_{\!\mathcal{T}}(S)+\beta z_i\nabl\lunk \phi\runk_{\!\mathcal{T}}^{k}(S)\right)\cdot\nabl\varphi^{k}(S).
	\end{align*}
	Using the Cauchy-Schwarz inequality, Lemma \ref{lem:electric} and the right inequality in \eqref{eq:entropone} in Lemma \ref{lem:entropy_estimate}, we get that
	\begin{multline*}
		\sum_{k=1}^{N_T}\sum_{K\in\mathcal{M}}|K|\left(u_{i,K}^{k}-u_{i,K}^{k-1}\right)\varphi_K^{k}\\
		\leq\ctel{cte:tildeCone}\left(\sum_{k=1}^{N_T}\tau^{k}\sum_{K\in\mathcal{M}}\sum_{S\in\mathcal{T}_K}|S|{u}_{i,S}^{k}{u}_{0,S}^{k} D_i\left(|\nabl\lunk\mu_i\runk^{k}_{\!\mathcal{T}}(S)|^2+\beta z_i|\nabl\lunk \phi\runk_{\!\mathcal{T}}^{k}(S)|^2\right)\right)^{\frac12}\\
		\times \left(\sum_{k=1}^{N_T} \tau^{k}\sum_{K\in\mathcal{M}}\sum_{S\in\mathcal{T}_K}|S||\nabl\lunk \phi\runk_{\!\mathcal{T}}^{k}(S)|^2\right)^{\frac12}
		\leq \ctel{cte:tildeC}\|\varphi\|_{L^2(0,T;H^1(\Omega))}.
	\end{multline*}
	The definition of the $H^{-1}_{\!\mathcal{M}}$-norm and the relation $\frac{u_{0,K}^{k}-u_{0,K}^{k-1}}{\tau^{k}} =  -\sum_{i=1}^n\frac{u_{i,K}^{k}-u_{i,K}^{k-1}}{\tau^{k}}$ conclude the claim.
\end{proof}
This result yields the following estimate for the time derivative of $\phi$.
\begin{lemma}\label{lemma:time_derivative_estimate_phi}
There exists $\ctel{cte:estimphit}>0$ such that, for any solution $((\mu_{i,K}^{k})_{i,K},(\phi_{K}^{k})_K)_k$ to  \eqref{eq:scheme}, we have
	\begin{align*}
		\frac 1 {C_P}\|\lunk\overline{\partial}_t  \phi\runk_{\!\mathcal{T}}\|_{L^2(0,T;L^2(\Omega))}\leq \| \nabl\lunk\overline{\partial}_t  \phi\runk_{\!\mathcal{T}}\|_{L^2(0,T;L^2(\Omega)^d)}\leq \cter{cte:estimphit}.
	\end{align*}
\end{lemma}
\begin{proof} Using \eqref{eq:phiD}-\eqref{eq:phidisc}, we get that $\overline{\partial}_t^{k} \lunk \phi\runk_{\!\mathcal{T}}\in V_{\!\mathcal{T},D}$ for $k=1,\ldots,{N_T}$, and
\[
	\lambda^2\int_\Omega\nabl\lunk\overline{\partial}_t \phi\runk^{k}_{\!\mathcal{T}}(x)\cdot\nabl v(x){\rm d}x=\int_\Omega \Big(\sum_{i=1}^n z_i \lunk\overline{\partial}_t u_i \runk^{k}_{\!\mathcal{M}}(x) \Big)v_{\!\mathcal{M}}(x){\rm d}x, \ \forall v\in V_{\!\mathcal{T},D}.
\]
Letting $v =  \lunk \overline{\partial}_t\phi\runk^{k}_{\!\mathcal{T}}$, we get
\[
	\lambda^2 \Vert \nabl\lunk\overline{\partial}_t \phi\runk^{k}_{\!\mathcal{T}}\Vert_{L^2(\Omega)^d}^2\le  \Big(\sum_{i=1}^n |z_i|\Vert  \lunk\overline{\partial}_t u_i \runk^{k}_{\!\mathcal{M}}\Vert_{H^{-1}_{\!\mathcal{M}}(\Omega)^d} \Big) \Vert \nabl\lunk\overline{\partial}_t \phi\runk^{k}_{\!\mathcal{T}}\Vert_{L^2(\Omega)^d}.
\]
Using Lemma \ref{lemma:time_derivative_estimate} yields the claim.
\end{proof}

The next lemma provides an estimate on the time translate of  $\lunk u_0 \runk_{\!\mathcal{M}}$ and  $\lunk u_0 \runk_{\!\mathcal{T}}$.
\begin{lemma} \label{lem:timtransuz}
Let $((\mu_{i,K}^{k})_{i,K},(\phi_{K}^{k})_K)_k$ be a solution to  \eqref{eq:scheme}. Then the functions  $\lunk u_0 \runk_{\!\mathcal{M}}$ and  $\lunk u_0 \runk_{\!\mathcal{T}}$ are such that
\begin{equation}\label{eq:timetransuz}
 \forall s\in [0,T],\ \Vert \lunk u_0 \runk_{\!\mathcal{T}}(\cdot+s) - \lunk u_0 \runk_{\!\mathcal{T}}\Vert_{L^2(0,T-s;L^2(\Omega))}^2\le \Vert \lunk u_0 \runk_{\!\mathcal{M}}(\cdot+s) - \lunk u_0 \runk_{\!\mathcal{M}}\Vert_{L^2(0,T-s;L^2(\Omega))}^2\le s\ctel{cte:timetransuz}.
\end{equation}

\end{lemma}
\begin{proof}
We write
\begin{multline*}
  \int_0^{T-s}\int_\Omega (\lunk u_0\runk_{\!\mathcal{M}}(t+s,x) - \lunk u_0 \runk_{\!\mathcal{M}}(t,x))^2 {\rm d}x{\rm d}t = \int_0^{T-s}\int_\Omega \lunk u_0 \runk_{\!\mathcal{M}}(t+s,x)(\lunk u_0\runk_{\!\mathcal{M}}(t+s,x) - \lunk u_0 \runk_{\!\mathcal{M}}(t,x))  {\rm d}x{\rm d}t\\
  -\int_0^{T-s}\int_\Omega \lunk u_0 \runk_{\!\mathcal{M}}(t,x)(\lunk u_0\runk_{\!\mathcal{M}}(t+s,x) - \lunk u_0 \runk_{\!\mathcal{M}}(t,x))   {\rm d}x{\rm d}t,
\end{multline*}
and we apply Lemma \ref{lem:esttimtran} to both terms and use Lemma \ref{lem:entropy_estimate} and Lemma \ref{lemma:time_derivative_estimate} to obtain the right inequality. The left inequality is then a consequence of \eqref{eq:equivmt} in Lemma \ref{lemma:equivalenz_of_H1L2norms}.
\end{proof}

We need a uniform bound on the $L^2$-norm in space and time of the gradient of $u_0u_i$ to obtain time translate estimates for $\sqrt{u_0}u_i$.
\begin{lemma} \label{lem:uzui}
Let $((\mu_{i,K}^{k})_{i,K},(\phi_{K}^{k})_K)_k$ be a solution to  \eqref{eq:scheme}. Then the function $\lunk u_0 u_i\runk_{\!\mathcal{T}} \in V_{\!\mathcal{M},\tau}$ is such that
\begin{equation}\label{eq:norhoneuzui}
 \Vert \nabl\lunk u_0 u_i\runk_{\!\mathcal{T}} \Vert_{L^2(0,T;L^2(\Omega))}\le \ctel{cte:norhoneuzui}.
\end{equation}

\end{lemma}
\begin{proof}
	We write for $i=1,\dots,n$ and $K,L\in\mathcal{M}_S$ for some $S\in\mathcal{T}$
 \[
	 u_{0,K}^{k} u_{i,K}^{k} - u_{0,L}^{k} u_{i,L}^{k} =\sqrt {u_{0,K}^{k}}u_{i,K}^k  \left(\sqrt {u_{0,K}^{k}}-\sqrt{u_{0,L}^k}\right) + \sqrt {u_{0,L}^{k}} \left(\sqrt {u_{0,K}^{k}}u_{i,K}^{k}-\sqrt {u_{0,L}^{k}}u_{i,L}^{k}\right).
 \]
 With $\sqrt{u_{0,K}^k}u_{i,K}^k, \sqrt{u_{0,L}^k}\in[0,1]$ follows
 \begin{multline*}
	 \sum_{S\in\mathcal{T}}\frac{|S|}{h_S^2}\sum_{\{K,L\}\in\mathcal{E}_S} (u_{0,K}^{k} u_{i,K}^{k} - u_{0,L}^{k} u_{i,L}^{k})^2 \\
	 \leq\sum_{S\in\mathcal{T}}\frac{|S|}{h_S^2}\sum_{\{K,L\}\in\mathcal{E}_S} 2\left(\left(\sqrt{u_{0,K}^{k}}u_{i,K}^{k}-\sqrt{u_{0,L}^{k}}u_{i,L}^{k}\right)^2+\left(\sqrt{u_{0,K}^{k}}-\sqrt{u_{0,L}^{k}}\right)^2\right).
 \end{multline*}
 With Lemma \ref{lem:entropy_estimate} follows the claim.

\end{proof}

We can now derive time translate estimates for $\sqrt{u_0}u_i$.
\begin{lemma}\label{lem:timtransuzi} Let $((\mu_{i,K}^{k})_{i,K},(\phi_{K}^{k})_K)_k$ be a solution to  \eqref{eq:scheme}.
Then, for all $i=1,\ldots,n$, the functions $\lunk \sqrt{u_0}  u_i\runk_{\!\mathcal{T}} \in V_{\!\mathcal{T},\tau}$ and  $\lunk \sqrt{u_0}  u_i\runk_{\!\mathcal{M}} \in X_{\!\mathcal{M},\tau}$  are such that
\begin{multline}\label{eq:timtransvi}
	\forall s\in [0,T],\ \Vert\lunk \sqrt{u_0}  u_i\runk_{\!\mathcal{T}}(\cdot+s) -\lunk \sqrt{u_0}  u_i\runk_{\!\mathcal{T}}\Vert_{L^2(0,T-s;L^2(\Omega))}^2\\
	\le \Vert\lunk \sqrt{u_0}  u_i\runk_{\!\mathcal{M}}(\cdot+s) -\lunk \sqrt{u_0}  u_i\runk_{\!\mathcal{M}}\Vert_{L^2(0,T-s;L^2(\Omega))}^2 \le \ctel{cte:timtransvi}\sqrt{s}.
\end{multline}
\end{lemma}
\begin{proof}
For a.e. $(t,x)\in (0,T-s)\times\Omega$, letting $a = \lunk u_0 \runk_{\!\mathcal{M}}(t,x)$, $b= \lunk u_i \runk_{\!\mathcal{M}}(t,x)$, $a' = \lunk u_0 \runk_{\!\mathcal{M}}(t+s,x)$, $b'= \lunk u_i \runk_{\!\mathcal{M}}(t+s,x)$, we have
\begin{align*}
	\left(\sqrt{a'}b'-\sqrt{a}b\right)^2
	&= a'b'^2-2\sqrt{aa'}bb'+ab^2
	= a'b'(b'-b)+ab(b-b')+2bb'(\frac {a+a'} 2 -\sqrt{aa'})\\
	&\leq \underbrace{a'b'(b'-b)+ab(b-b')}_{=:A_1(t,x)}+\underbrace{2|a-a'|}_{=:A_2(t,x)}.
\end{align*}
Lemma \ref{lem:esttimtran} yields
\begin{align*}
	\int_0^{T-s}\int_\Omega A_{1}(t,x)\,{\rm d}x{\rm d}t\leq 2s\|\lunk\overline{\partial}_t u_i\runk_{\!\mathcal{M}}\|_{L^2(0,T;H^{-1}_{\!\mathcal{M}}(\Omega))}\|\nabl \lunk u_0u_i\runk_{\!\mathcal{T}}\|_{L^2(0,T;L^2(\Omega)^d)}.
\end{align*}
With Lemma \ref{lem:uzui} and Lemma \ref{lemma:time_derivative_estimate} we further estimate
\begin{align}\label{eq:A11_bound}
	\int_0^{T-s}\int_\Omega A_{1}(t,x){\rm d}x{\rm d}t\leq 2C_{12}C_7s.
\end{align}
To obtain a bound on the integral over $A_2(t,x)$ we use Lemma \ref{lem:timtransuz} and the Cauchy-Schwarz inequality
\begin{align}\label{eq:A12_13_2_bound}
	\int_0^{T-s}\int_\Omega A_{2}(t,x){\rm d}x{\rm d}t \leq 2\sqrt{C_{11}}\sqrt{|\Omega|(T-s)}\sqrt{s}
\end{align}
Combining \eqref{eq:A11_bound} and \eqref{eq:A12_13_2_bound} yields the claim.
\end{proof}
The next lemma will be useful for the convergence study.
\begin{lemma}\label{lem:difminmax}
 Let $((\mu_{i,K}^{k})_{i,K},(\phi_{K}^{k})_K)_k$ be a solution to  \eqref{eq:scheme}.
 Define for all $k$, for $i=0,\ldots,n$ and $S\in\mathcal{T}$, the values $\overline{u}_{i,S}^{k} = \max_{K\in \mathcal{M}_S} u_{i,K}^{k}$ and $\underline{u}_{i,S}^{k} = \min_{K\in \mathcal{M}_S} u_{i,K}^{k}$.
 Then the following inequalities hold
 \begin{align}\label{eq:difminmax}
	 \begin{split}
	 \sqrt{\overline{u}_{0,S}^{k}}-\sqrt{\underline{u}_{0,S}^{k}}&\leq 2|\mathcal{T}||\nabl\lunk\sqrt{u_0}\runk_{\!\mathcal{T}}^k(S)|,\\
	 \sqrt{\overline{u}_{0,S}^{k}}\overline{u}_{i,S}^{k}-\sqrt{\underline{u}_{0,S}^{k}}\underline{u}_{i,S}^{k}&\leq |\mathcal{T}|\left(2|\nabl\lunk\sqrt{u_0}\runk_{\!\mathcal{T}}^k(S)|+|\nabl\lunk\sqrt{u_0}u_i\runk_{\!\mathcal{T}}^k(S)|\right),\\
	 \overline{u}_{0,S}^{k}\overline{u}_{i,S}^{k}-\underline{u}_{0,S}^{k}\underline{u}_{i,S}^{k}&\leq |\mathcal{T}|\left(3|\nabl\lunk\sqrt{u_0}\runk_{\!\mathcal{T}}^k(S)|+|\nabl\lunk\sqrt{u_0}u_i\runk_{\!\mathcal{T}}^k(S)|\right).
	 \end{split}
 \end{align}
\end{lemma}
\begin{proof}
	We have, for $i=0,\ldots,n$, $\overline{u}_{i,S}^{k} =  u_{i,\overline{K}_i^{k}}^{k}$ and  $\underline{u}_{i,S}^{k} =  u_{i,\underline{K}_i^{k}}^{k}$, with $\overline{K}_i^{k}, \underline{K}_i^{k}\in \mathcal{M}_S$. We can than write, for $i=1,\ldots,n$,
	\begin{multline*}
		\sqrt{\overline{u}_{0,S}^{k}}\overline{u}_{i,S}^{k} - \sqrt{\underline{u}_{0,S}^{k}}\underline{u}_{i,S}^{k}
		=\underbrace{u_{i,\overline{K}^k_i}^{k}\left(\sqrt{u_{0,\overline{K}_0^k}^{k}}-\sqrt{u_{0,\overline{K}_i^k}^{k}}\right)
		+u_{i,\underline{K}_i^k}^k\left(\sqrt{u_{0,\underline{K}_i^k}^{k}}-\sqrt{u_{0,\underline{K}_0^k}^{k}}\right)}_{=:I_1}\\
		+\underbrace{\sqrt{u_{0,\overline{K}_i^k}^k}u_{i,\overline{K}_i^k}^k-\sqrt{u_{0,\underline{K}_i^k}^k}u_{i,\underline{K}_i^k}^k}_{=:I_2}.
	\end{multline*}
	To estimate $I_1$, we use that $0\leq u_i\leq 1$ for all $i=0,\dots,n$, which provides
	\begin{align*}
		|I_1|\leq
		\left|\sqrt{u_{0,\overline{K}_0^{k}}^{k}}-\sqrt{u_{0,\overline{K}_i^{k}}^{k}}\right|+\left|\sqrt{u_{0,\underline{K}_0^{k}}^{k}}-\sqrt{u_{0,\underline{K}_i^{k}}^{k}}\right|.
	\end{align*}
	By the definition of piecewise affine functions, we have
	\begin{align*}
		|I_1| \leq 2|\mathcal{T}|  |\nabl \lunk \sqrt{u_0} \runk_{\!\mathcal{T}}^{k}(S)|.
	\end{align*}
	Similarly, we obtain for $I_2$ the bound
	\begin{align*}
		|I_2|\leq |\mathcal{T}||\nabl\lunk \sqrt{u_0}u_i\runk_{\!\mathcal{T}}^k(S)|.
	\end{align*}
	Combining the above yields
	\begin{align*}
		|I_1|+|I_2|\leq 2|\mathcal{T}||\nabl \lunk \sqrt{u_0} \runk_{\!\mathcal{T}}^{k}(S)|+|\mathcal{T}|^2|\nabl\lunk \sqrt{u_0}u_i\runk_{\!\mathcal{T}}^k(S)|.
	\end{align*}
	Lemma \ref{lem:entropy_estimate} concludes the proof for  $\sqrt{\overline{u}_{0,S}^{k}}\overline{u}_{i,S}^{k}-\sqrt{\underline{u}_{0,S}^{k}}\underline{u}_{i,S}^{k}$.
	The bound for $\sqrt{\overline{u}_{0,S}^{k}}-\sqrt{\underline{u}_{0,S}^{k}}$ follows similar to the bound of $|I_2|$.\\
	The bound for $\overline{u}_{0,S}^{k}\overline{u}_{i,S}^{k}-\underline{u}_{0,S}^{k}\underline{u}_{i,S}^{k}$ from the above bounds via
	\begin{align*}
		u_{0,\overline{K}_0^{k}}^{k}u_{i,\overline{K}_i^{k}}^{k}-u_{0,\underline{K}_0^{k}}^{k}u_{i,\underline{K}_i^{k}}^{k}
		&=\sqrt{u_{0,\overline{K}_0^{k}}^{k}}u_{i,\overline{K}_i^{k}}^{k}\left(\sqrt{u_{0,\overline{K}_0^{k}}^{k}}-\sqrt{u_{0,\underline{K}_0^k}^k}\right)
		+\sqrt{u_{0,\underline{K}_0^k}^k}\left(\sqrt{u_{0,\overline{K}_0^k}^k}u_{i,\overline{K}_i^k}-\sqrt{u_{0,\underline{K}_0^k}^k}u_{i,\underline{K}_i^k}^k\right)\\
		& \leq |\mathcal{T}|\left(3|\nabl\lunk\sqrt{u_0}\runk_{\!\mathcal{T}}^k(S)|+|\nabl\lunk\sqrt{u_0}u_i\runk_{\!\mathcal{T}}^k(S)|\right).
	\end{align*}
\end{proof}

\section{Convergence study}\label{sec:convscheme}

Let $\theta>0$, and let $(\mathcal{T}_m,\tau_m)_{m\in\mathbb{N}}$ be a sequence of simplicial meshes and time discretisations such that
\begin{align*}
	\lim_{m\to\infty} |\mathcal{T}_m|=\lim_{m\to\infty} |\tau_m| =0 \hbox{ and }\sup_m \theta_{\!\mathcal{T}_m} \le \theta.
\end{align*}
As stated above, we assume that $V_{\!\mathcal{T}_m,D}\subset H^1_D(\Omega)$ and we assume that the orthogonal projection $\mathcal{P}_m: H^1_D(\Omega)\to V_{\!\mathcal{T}_m,D}$ in $H^1(\Omega)$ is such that 
\begin{equation}\label{eq:approxvtminHone}
\forall v\in H^1_D(\Omega),\ \lim_{m\to\infty} \Vert v - \mathcal{P}_m(v)\Vert_{H^1(\Omega)} = 0.
\end{equation}
Moreover, we assume the existence of $\phi^D_{\!\mathcal{T}_m}\in V_{\!\mathcal{T}}$, which converges to  $\phi^D$ in $H^1(\Omega)$. This implies that there exists a constant $C_\phi^D>0$, independent of the discretisation, such that
\begin{equation}\label{eq:cphiDm}
 \forall m\in\mathbb{N},\  \Vert \phi^D_{\!\mathcal{T}_m}\Vert_{H^1(\Omega)}\le C_\phi^D,
\end{equation}
which makes \eqref{eq:cphiD} hold.

\medskip

For each $m\in\mathbb{N}$, let $\mathcal{M}_m$ be the dual mesh associated to the simplicial mesh $\mathcal{T}_m$. Letting $\mathcal{T} = \mathcal{T}_m$ and $\tau = \tau_m$, we denote by $\lunk u_i \runk_{\!\mathcal{T}_m}$, $\lunk u_i \runk_{\!\mathcal{M}_m}$, $\lunk \phi\runk_{\!\mathcal{T}_m}$ the functions obtained by Definition \ref{def:appsol} from a solution $((\mu_{i,K}^{k})_{i,K},(\phi_{K}^{k})_K)_k$ to  \eqref{eq:scheme}, which therefore fulfill the bounds derived in Section \ref{sec:anascheme}. This allows to state the following lemma.

\begin{lemma}\label{lemma:compactness}
	There exists a subsequence of $(\mathcal{T}_m,\tau_m)_{m\in\mathbb{N}}$, that is again denoted by $(\mathcal{T}_m,\tau_m)_{m\in\mathbb{N}}$ and there exist:
	\begin{itemize}
	 \item functions $u_i\in L^\infty((0,T)\times \Omega)$ for $i=0,\ldots,n$ with $u_0,\sqrt{u_0},\sqrt{u_0}u_i\in L^2(0,T;H^1(\Omega))$,
	 \item a function $\phi\in L^2(0,T;H^1_D(\Omega)+ \phi^D)$,
	\end{itemize}
  such that the following holds.
  \begin{enumerate}[(i)]
	\item\label{eqn:compactness_i} for all $i=1,\ldots,n$,  $\lunk u_i \runk_{\!\mathcal{M}_m}$ converges to $u_i$ for the weak-$\star$ topology of $L^\infty((0,T)\times \Omega)$,
	\item\label{eqn:compactness_0}  $\lunk u_0 \runk_{\!\mathcal{M}_m}$ and   $\lunk u_0 \runk_{\!\mathcal{T}_m}$ converge to $u_0$ in $L^2((0,T)\times \Omega)$,  $\lunk u_0 \runk_{\!\mathcal{T}_m}$ weakly converges to $u_0$ in $L^2(0,T;H^1(\Omega))$ and $\lunk \sqrt{u_0} \runk_{\!\mathcal{T}_m}$ weakly converges to $\sqrt{u_0}$ in $L^2(0,T;H^1(\Omega))$,
	\item\label{eqn:compactness_isqrt0} $\lunk \sqrt{u_0}u_i \runk_{\!\mathcal{M}_m}$ and $\lunk \sqrt{u_0}u_i \runk_{\!\mathcal{T}_m}$ converge to $\sqrt{u_{0}}u_i$ in $L^2((0,T)\times \Omega)$, and  $\lunk \sqrt{u_0}u_i \runk_{\!\mathcal{T}_m}$ weakly converges to $\sqrt{u_{0}}u_i$ in $L^2(0,T;H^1(\Omega))$,
	\item\label{eqn:compactness_phi} $\lunk \phi\runk_{\!\mathcal{T}_m}$  converges to $\phi$ in $L^2((0,T)\times \Omega)$ and weakly in $L^2(0,T;H^1(\Omega))$.
	\end{enumerate}
\end{lemma}
\begin{proof}
Up to consecutive extractions of subsequences, we can state the following.

\medskip

\begin{enumerate}[(i)]
	\item Since $\lunk u_i \runk_{\!\mathcal{M}_m}(t,x)\in [0,1]$ for a.e. $(t,x)\in(0,T)\times\Omega$, we deduce from the Banach-Alaoglu Theorem the existence of functions $u_i\in L^\infty((0,T)\times \Omega)$ for $i=0,\ldots,n$ such that $\lunk u_i \runk_{\!\mathcal{M}_m}$ converges to $u_i$ for the weak-$\star$ topology of $L^\infty((0,T)\times \Omega)$.
	\item Using the bound \eqref{eq:entroptwo} and the time translate estimate in Lemma \ref{lem:timtransuz} on $\left\|\nabl \lunk u_0 \runk_{\!\mathcal{T}_m}\right\|_{L^2((0,T)\times\Omega)}$, we get
		from Kolmogorov theorem (see \cite[Theorem 4.26]{brezis2011}) that  $\lunk u_0 \runk_{\!\mathcal{T}_m}$ converges in $L^2((0,T)\times \Omega)$ and weakly in $L^2(0,T;H^1(\Omega))$. From \eqref{eq:equivmt} in Lemma \ref{lemma:equivalenz_of_H1L2norms}, we get that  $\lunk u_0 \runk_{\!\mathcal{M}_m}$ also converges in $L^2((0,T)\times \Omega)$ to the same limit, which is therefore equal to $u_0\in L^2(0,T;H^1(\Omega))$.
	From the convergence of  $\lunk u_0 \runk_{\!\mathcal{M}_m}$, we deduce that $\lunk \sqrt{u_0} \runk_{\!\mathcal{M}_m}$ converges in $L^2((0,T)\times \Omega)$ to $\sqrt{u_0}$ and using \eqref{eq:equivmt}, we get that $\lunk \sqrt{u_0} \runk_{\!\mathcal{T}_m}$ converges as well to the same limit. Then, using the bound \eqref{eq:entroptwo}, we get that  $\lunk \sqrt{u_0} \runk_{\!\mathcal{T}_m}$ weakly converges in $L^2(0,T;H^1(\Omega))$ to $\sqrt{u_0}\in L^2(0,T;H^1(\Omega))$.
	\item 
	Using that $\lunk u_i \runk_{\!\mathcal{M}_m}$ converges to $u_i$ for the weak-$\star$ topology of $L^\infty((0,T)\times \Omega)$ and the strong convergence of $\lunk \sqrt{u_0} \runk_{\!\mathcal{M}_m}$, we get that $\lunk \sqrt{u_0}u_i \runk_{\!\mathcal{M}_m}$ weakly converges to $\sqrt{u_{0}} u_i$ for the weak-$\star$ topology of $L^\infty((0,T)\times \Omega)$. Using  the bound \eqref{eq:entroptwo} and the time translate estimate in Lemma \ref{lem:timtransuzi}, we get that $\lunk \sqrt{u_0}u_i \runk_{\!\mathcal{T}_m}$  converges in $L^2((0,T)\times \Omega)$ and weakly in $L^2(0,T;H^1(\Omega))$. Again applying  \eqref{eq:equivmt}, we get that $\lunk \sqrt{u_0}u_i \runk_{\!\mathcal{M}_m}$ also  converges in $L^2((0,T)\times \Omega)$, therefore to the same limit, $\sqrt{u_{0}} u_i\in L^2(0,T;H^1(\Omega))$.
	\item 
	Using the bound in Lemma \ref{lem:electric} and the time translate estimate in Lemma \ref{lemma:time_derivative_estimate_phi} and using the convergence of  $\phi^D_{\!\mathcal{T}_m}\in V_{\!\mathcal{T}}$ to  $\phi^D$ in $H^1(\Omega)$, we get that $\lunk \phi\runk_{\!\mathcal{T}_m}$ converges in $L^2((0,T)\times \Omega)$ and weakly in $L^2(0,T;H^1(\Omega))$ to $\phi\in L^2(0,T;H^1_D(\Omega)+ \phi^D)$.
\end{enumerate}
\end{proof}

The main challenge of the convergence proof is to switch from the entropy variable formulation of the scheme \eqref{eq:scheme} to the formulation in the ion concentrations and solvent concentrations. This is completed in the proof of the following theorem.

\begin{thm}\label{thm:convtosol} Let  $(\mathcal{T}_m,\tau_m)_{m\in\mathbb{N}}$ be a sequence of discretisations following the hypotheses of this section. Again denote by $(\mathcal{T}_m,\tau_m)_{m\in\mathbb{N}}$ the subsequence  given by Lemma \ref{lemma:compactness}. Let $(u_0,\dots,u_n,\phi)$ be the functions, the existence of which is also provided by Lemma \ref{lemma:compactness}. Then  $(u_0,\dots,u_n,\phi)$ is a weak solution to System \eqref{eq:contmod} in the sense of Definition \ref{def:weaksolcont}.
\end{thm}
\begin{proof}
The proof of \eqref{eq:contspacesi} is done in Lemma \ref{lemma:compactness}.
The proof that \eqref{eq:weakcontphi} holds is obtained for $v\in H^1_D(\Omega)$ and $\xi \in C^\infty_c((0,T))$, by considering $\mathcal{T} := \mathcal{T}_m$, multiplying \eqref{eq:phidisc} by  $v_K\int_{t^{k-1}}^{t^{k}} \xi(t){\rm d}t$ where $v_{\mathcal{T}_m} = \mathcal{P}_m(v)$ (see \eqref{eq:approxvtminHone}), by summing on $K$ and $k$, and by passing to the limit $m\to +\infty$.

\medskip

Let us now turn to the proof that \eqref{eq:weakconti} holds.
Let $\psi\in C_c^{\infty}([0,T)\times\mathbb{R}^d)$. Let $i=1,\ldots,n$ and $m\in\mathbb{N}$. For the simplicity of notation, we drop the index $m$ in the discrete quantities involved by $\mathcal{T}_m$, $\mathcal{M}_m$ and $\tau_m$, and we use as well the notation $\mathcal{T} := \mathcal{T}_m$, $\mathcal{M}:=\mathcal{M}_m$ and $\tau :=\tau_m$. Define $\psi_{\!\mathcal{M}}\in X_{\!\mathcal{M},\tau}$ and $\psi_{\!\mathcal{T}}\in V_{\!\mathcal{T},\tau}$ via
	\begin{align*}
		\psi_{K}^{k}:=\psi(t^{k-1},x_K)\quad\forall k=1,\dots,{N_T}+1\text{ and }K\in\mathcal{M}.
	\end{align*}
	Multiplying \eqref{eq:ui} with $\tau^k\psi^{k}_K$ and summing over $K\in\mathcal{M}$ and $k=1,\dots,{N_T}$ yield
	\begin{align}\label{eqn:conv_multiplied}
		T_m^{(0)}+D_i T_m^{(1)}+D_i \beta z_iT_m^{(2)}=0,
	\end{align}
	where $T_m^{(0)}$, $T_m^{(1)}$ and $T_m^{(2)}$ are defined by
	\begin{align}\label{eqn:def}
		T_m^{(0)} := \sum_{k=1}^{{N_T}} \sum_{K\in\mathcal{M}}|K|(u_{i,K}^{k}-u_{i,K}^{k-1})\psi_K^{k},
	\end{align}
\begin{align}\label{eqn:def_T_1}
  T^{(1)}_m = \sum_{k=1}^{N_T} \tau^{k} \sum_{S\in \mathcal{T}} {u}_{i,S}^{k} {u}_{0,S}^{k}\sum_{\{K,L\}\in\mathcal{E}_S}  a_{KL}^S(\mu_{i,K}^{k}-\mu_{i,L}^{k}) (\psi_K^{k} - \psi_L^{k}),
\end{align}
and
 \begin{align}\label{eqn:def_T_0}
  T^{(2)}_m := \sum_{k=1}^{N_T} \tau^{k} \sum_{S\in \mathcal{T}} {u}_{i,S}^{k} {u}_{0,S}^{k}\sum_{\{K,L\}\in\mathcal{E}_S} a_{KL}^S(\phi_K^{k}-\phi_L^{k}) (\psi_K^{k} - \psi_L^{k}).
\end{align}

{\bf Limit of the term $T_m^{(0)}$.}

We have, accounting for $\psi_K^{N_T+1} =0$,
	\begin{align}\label{eqn:conv_temporal}
		T_m^{(0)}
		=-\sum_{k=1}^{{N_T}} \sum_{K\in\mathcal{M}}|K|u_{i,K}^{k}(\psi_K^{k+1}-\psi_K^{k})-\sum_{K\in\mathcal{M}}|K|u_{i,K}^0\psi_K^1.
	\end{align}
	We define $\overline{\partial}_t\psi_{\!\mathcal{M}}\in X_{\!\mathcal{M},\tau}$ by the value $\frac {\psi_K^{k+1}-\psi_K^{k}}{\tau^k}$ at the point $(t^k,x_K)$ (note that the time index is shifted compared to the discrete time derivatives of the unknown functions). This yields
	\[
	 T_m^{(0)} = - \int_0^T\int_\Omega \lunk u_i \runk_{\!\mathcal{M}_m}(t,x)\overline{\partial}_t\psi_{\!\mathcal{M}_m}(t,x){\rm d}x{\rm d}t - \int_\Omega u_{\!\mathcal{M}_m}^0(x)\psi_{\!\mathcal{M}_m}(0,x){\rm d}x.
	\]
The smoothness properties of $\psi$, the definition \eqref{eq:uiini} of $u_K^0$ and the weak convergence of $\lunk u_i \runk_{\!\mathcal{M}_m}$ to $u_i$ lead to
\begin{equation}\label{eqnconvtt}
 \lim_{m\to +\infty} T_m^{(0)} =  - \int_0^T\int_\Omega u_{i}(t,x){\partial}_t\psi(t,x){\rm d}x{\rm d}t - \int_\Omega u^{\rm ini}(x)\psi(0,x){\rm d}x.
\end{equation}

{\bf Limit of the term $T_m^{(1)}$.}

We first observe that the term
\begin{align*}
T_m^{(10)} := \int_0^T\int_\Omega (\lunk \sqrt{u_0}\runk_{\!\mathcal{M}_m} \nabl \lunk \sqrt{u_0}u_i\runk_{\!\mathcal{T}_m}-3 \lunk u_i \sqrt{u_0}\runk_{\!\mathcal{M}_m}\nabl\lunk \sqrt{u_0}\runk_{\!\mathcal{T}_m})\cdot\nabl\psi_{\!\mathcal{T}_m}\,{\rm d}x\,{\rm d}t
\end{align*}
satisfies, by the weak/strong convergence results proved in Lemma \ref{lemma:compactness},
\begin{align*}
\lim_{m\to +\infty} T_m^{(10)} = \int_0^T\int_\Omega (\sqrt{u_0}\nabl (\sqrt{u_0}u_i)-3u_i\sqrt{u_0}\nabl\sqrt{u_0})\cdot\nabl\psi\,{\rm d}x\,{\rm d}t.
\end{align*}
Note that
\begin{multline*}
	T^{(10)}_m 
	= \sum_{k=1}^{N_T} \tau^{k} \sum_{S\in \mathcal{T}} \left(\left( \frac1{d+1}\sum_{K\in\mathcal{M}_S}\sqrt{u_{0,K}^{k}}\right)\sum_{\{K,L\}\in\mathcal{E}_S} a_{KL}^S \left(\sqrt{u_{0,K}^{k}} u_{i,K}^{k} - \sqrt{u_{0,L}^{k}} u_{i,L}^{k}\right) (\psi_K^{k} - \psi_L^{k})\right.\\ 
	\left.- 3 \left(\frac1{d+1} \sum_{K\in\mathcal{M}_S}u_{i,K}^{k}\sqrt{u_{0,K}^{k}}\right)\sum_{\{K,L\}\in\mathcal{E}_S} a_{KL}^S \left(\sqrt{u_{0,K}^{k}} - \sqrt{u_{0,L}^{k}}\right)(\psi_K^{k} - \psi_L^{k})\right) .
\end{multline*}
We first compare $T_m^{(10)}$ with
\begin{align}\label{eqn:def_T_2}
T^{(11)}_m = \sum_{k=1}^{N_T} \tau^{k} \sum_{S\in \mathcal{T}} \sum_{\{K,L\}\in\mathcal{E}_S}  a_{KL}^S(u_{i,K}^{k} u_{0,L}^{k} - u_{i,L}^{k} u_{0,K}^{k}) (\psi_K^{k} - \psi_L^{k}).
\end{align}
From Lemma \ref{lemma:useful_inequalities_cvtwo} by letting $b=u_{0,K}^{k}$ and $b'=u_{0,L}^{k}$ follows the existence of $\widehat{u}_{0,KL}^{k},\check{u}_{0,KL}^{k}\in I(b,b')$ such that
\begin{multline*}
T^{(11)}_m = \sum_{k=1}^{N_T} \tau^{k} \sum_{S\in \mathcal{T}} \sum_{\{K,L\}\in\mathcal{E}_S} a_{KL}^S\Big(\sqrt{\widehat{u}_{0,KL}^{k}}(\sqrt{u_{0,K}^{k}} u_{i,K}^{k} - \sqrt{u_{0,L}^{k}} u_{i,L}^{k})\\ - 3 \frac {u_{i,K}^{k}+u_{i,L}^{k}} 2 \sqrt{\check{u}_{0,KL}^{k}} (\sqrt{u_{0,K}^{k}} - \sqrt{u_{0,L}^{k}})\Big) (\psi_K^{k} - \psi_L^{k}).
\end{multline*}
Hence we get, from the  Cauchy-Schwarz inequality,
\[
 (T^{(11)}_m - T^{(10)}_m)^2\le T^{(101)}_m T^{(102)}_m,
\]
with
\begin{multline*}
T^{(101)}_m = \sum_{k=1}^{N_T} \tau^{k} \sum_{S\in \mathcal{T}} \sum_{\{K,L\}\in\mathcal{E}_S} |a_{KL}^S| (\psi_K^{k} - \psi_L^{k})^2\\
\times \left(\left| \frac1{d+1}\sum_{K\in\mathcal{M}_S}\sqrt{u_{0,K}^{k}} - \sqrt{\widehat{u}_{0,KL}^{k}} \right|^2 + 3 \left|\frac1{d+1} \sum_{K\in\mathcal{M}_S}u_{i,K}^{k}\sqrt{u_{0,K}^{k}} - \frac {u_{i,K}^{k}+u_{i,L}^{k}} 2 \sqrt{\check{u}_{0,KL}^{k}}\right|^2\right),
\end{multline*}
and
\begin{equation*}
T^{(102)}_m = \sum_{k=1}^{N_T} \tau^{k} \sum_{S\in \mathcal{T}} \sum_{\{K,L\}\in\mathcal{E}_S} |a_{KL}^S| \left((\sqrt{u_{0,K}^{k}} u_{i,K}^{k} - \sqrt{u_{0,L}^{k}} u_{i,L}^{k})^2 + 3 (\sqrt{u_{0,K}^{k}} - \sqrt{u_{0,L}^{k}})^2\right).
\end{equation*}
Using Lemma \ref{lem:difminmax}, we remark that
\[
	\left| \frac1{d+1}\sum_{K\in\mathcal{M}_S}\sqrt{u_{0,K}^{k}} - \sqrt{\widehat{u}_{0,KL}^{k}} \right| \le |\mathcal{T}||\nabl \lunk\sqrt{u_0}\runk_{\!\mathcal{T}}^k(S)|,
\]
and that
\[
	\left|\frac1{d+1} \sum_{K\in\mathcal{M}_S}u_{i,K}^{k}\sqrt{u_{0,K}^{k}} - \frac {u_{i,K}^{k}+u_{i,L}^{k}} 2 \sqrt{\check{u}_{0,KL}^{k}}\right|  \le |\mathcal{T}|\left(2|\nabl \lunk\sqrt{u_0}\runk_{\!\mathcal{T}}^k(S)|+|\nabl \lunk\sqrt{u_0}u_i\runk_{\!\mathcal{T}}^k(S)|\right).
\]
From the Cauchy-Schwarz inequality, using $|\psi_K^{k} - \psi_L^{k}|\le h_S\Vert\nabl\psi\Vert_\infty$ and using
 \[
  |a_{KL}^S| \le \theta_{\mathcal{T}}^2\frac {|S|} {h_S^2},
 \]
 we get
\begin{equation*}
T^{(101)}_m \le |\mathcal{T}|^2\Vert\nabl\psi\Vert_\infty^2\theta_{\mathcal{T}}^2 \frac {d(d+1)} 2 \Big(\Vert\nabl \lunk\sqrt{u_0}\runk_{\!\mathcal{T}}\Vert_{L^2}^2 + 9\big(2 \Vert\nabl \lunk\sqrt{u_0}\runk_{\!\mathcal{T}}\Vert_{L^2}^2+\Vert\nabl \lunk\sqrt{u_0}u_i\runk_{\!\mathcal{T}}\Vert_{L^2}^2\big)\Big),
\end{equation*}
and, using \eqref{eq:gradt},
\begin{equation*}
 T^{(102)}_m \le \frac 1 {C_{\rm min}} \theta_{\mathcal{T}}^2  \Big(\Vert\nabl\lunk \sqrt{u_0}u_i\runk_{\!\mathcal{T}_m}\Vert_{L^2}^2  +
 3 \Vert\nabl\lunk \sqrt{u_0} \runk_{\!\mathcal{T}_m}\Vert_{L^2}^2 \Big).
\end{equation*}
Applying Lemma \ref{lem:entropy_estimate} proves that $T^{(101)}_m$ tends to 0 and that  $T^{(102)}_m$ remains bounded, so that
 $\lim_{m\to +\infty}  (T^{(11)}_m -T^{(10)}_m) =0$, which yields
\begin{align*}
\lim_{m\to +\infty} T_m^{(11)} = \int_0^T\int_\Omega (\sqrt{u_0}\nabl (\sqrt{u_0}u_i)-3u_i\sqrt{u_0}\nabl\sqrt{u_0})\cdot\nabl\psi\,{\rm d}x\,{\rm d}t.
\end{align*}
We now compare  $T_m^{(11)}$ with $T_m^{(1)}$. We first observe that, defining $\tilde{u}_{i,K,L}^{k}$ and $\tilde{u}_{0,K,L}^{k}$  by the values $\tilde{a}$ and $\tilde{b}$, provided by Lemma \ref{lemma:useful_inequalities_cvone} letting $b = u_{0,K}^{k}$, $b' = u_{0,L}^{k}$, $a = u_{i,K}^{k}$, $a' = u_{i,L}^{k}$, we can write that
\[
 T^{(11)}_m =  \sum_{k=1}^{({N_T})_m} \tau^{k}\sum_{S\in \mathcal{T}_m}  \sum_{L\in \mathcal{M}_{S}}\sum_{K\in \mathcal{M}_{S}}\tilde{u}_{i,K,L}^{k}\tilde{u}_{0,K,L}^{k}  (\mu_{i,K}^{k}-\mu_{i,L}^{k})(\psi_K^{k}-\psi_L^{k}).
\]
Hence we get from the Cauchy-Schwarz inequality that
\[
 (T^{(1)}_m - T^{(11)}_m )^2 \le  T^{(12)}_m T^{(13)}_m,
 \]
 with
 \[
  T^{(12)}_m  = \sum_{k=1}^{({N_T})_m} \tau^{k}\sum_{S\in \mathcal{T}_m} \overline{u}_{0,S}^{k}  \overline{u}_{i,S}^{k}  \sum_{L\in \mathcal{M}_{S}}\sum_{K\in \mathcal{M}_{S}}  |a_{KL}^S|(\mu_{i,K}^{k}-\mu_{i,L}^{k})^2,
 \]
 and
 \[
  T^{(13)}_m = \sum_{k=1}^{({N_T})_m} \tau^{k}\sum_{S\in \mathcal{T}_m}  \sum_{L\in \mathcal{M}_{S}}\sum_{K\in \mathcal{M}_{S}}  |a_{KL}^S| \frac {({u}_{0,S}^{k}  {u}_{i,S}^{k} - \tilde{u}_{i,K,L}^{k}\tilde{u}_{0,K,L}^{k})^2 }{\overline{u}_{0,S}^{k}  \overline{u}_{i,S}^{k}} (\psi_K^{k}-\psi_L^{k})^2.
 \]
 Using
 \[
  |a_{KL}^S| \le \theta_{\mathcal{T}}^2\frac {|S|} {h_S^2},
 \]
 we get that $T^{(12)}_m$ is bounded owing to Lemma \ref{lem:entropy_estimate} and \eqref{eq:uzt}.
We now remark that, from the properties of $\tilde{u}_{i,K,L}^{k}$ and $\tilde{u}_{0,K,L}^{k}$ proved in Lemma \ref{lemma:useful_inequalities_cvone} and using Lemma \ref{lem:difminmax},
 \begin{align*}
  \frac {({u}_{0,S}^{k}  {u}_{i,S}^{k} - \tilde{u}_{i,K,L}^{k}\tilde{u}_{0,K,L}^{k})^2 }{\overline{u}_{0,S}^{k}  \overline{u}_{i,S}^{k}}
  \le \overline{u}_{0,S}^{k}  \overline{u}_{i,S}^{k} - \underline{u}_{i,S}^{k}\underline{u}_{0,S}^{k}
  \leq |\mathcal{T}|\left(2|\nabl\lunk \sqrt{u_0}\runk_{\!\mathcal{T}_m}(S)|+|\nabl\lunk \sqrt{u_0}u_i\runk_{\!\mathcal{T}_m}(S)|\right)
 \end{align*}
Since $ (\psi_K^{k}-\psi_L^{k})^2 \le \Vert\nabl\psi\Vert_\infty^2 h_S^2$,  owing to  the Cauchy-Schwarz inequality, we get that
\[
	|T^{(13)}_m|^2 \le C |\mathcal{T}|^2\left(\Vert \nabl\lunk \sqrt{u_0}u_i\runk_{\!\mathcal{T}_m} \Vert_{L^2}^2+\Vert \nabl\lunk \sqrt{u_0}\runk_{\!\mathcal{T}_m} \Vert_{L^2}^2\right),
\]
where $C$ only depends on $\theta_{\mathcal{T}}$, $\Vert\nabl\psi\Vert_\infty$, $|\Omega|$, $T$ and $d$. Hence $T^{(13)}_m$ tends to zero, again applying Lemma \ref{lem:entropy_estimate}. Therefore we have
 $\lim_{m\to +\infty}  (T^{(1)}_m -T^{(11)}_m) =0$, which yields
\begin{align*}
\lim_{m\to +\infty} T_m^{(1)} = \int_0^T\int_\Omega (\sqrt{u_0}\nabl (\sqrt{u_0}u_i)-3u_i\sqrt{u_0}\nabl\sqrt{u_0})\cdot\nabl\psi\,{\rm d}x\,{\rm d}t.
\end{align*}

{\bf Limit of the term $T_m^{(2)}$.}

We observe that the term
\[
  T^{(21)}_m := \int_0^T\int_\Omega \lunk u_0 \runk_{\!\mathcal{M}_m}  \lunk u_i \runk_{\!\mathcal{M}_m} \nabl \lunk \phi\runk_{\!\mathcal{T}_m}\cdot \nabl \psi_{\!\mathcal{T}_m}{\rm d}x{\rm d}t
\]
satisfies
\[
 \lim_{m\to +\infty}  T^{(21)}_m =  \int_0^T\int_\Omega u_{0}  u_{i} \nabl \phi\cdot \nabl \psi {\rm d}x{\rm d}t.
\]
Indeed,   $\lunk u_0 \runk_{\!\mathcal{M}_m}  \lunk u_i \runk_{\!\mathcal{M}_m}=  \lunk  \sqrt{u_0}\runk_{\!\mathcal{M}_m} \lunk u_i \runk_{\!\mathcal{M}_m} \times \lunk  \sqrt{u_0}\runk_{\!\mathcal{M}_m}$ converges in $L^2$ to $\sqrt{u_{0}} u_{i} \times \sqrt{u_{0}}$, $\nabl \lunk \phi\runk_{\!\mathcal{T}_m}$ weakly converges in $L^2$ to  $\nabl \phi$ and $\nabl \psi_{\!\mathcal{T}_m}$ converges in $L^2$ to  $\nabl \psi$. We can write
\[
	T^{(21)}_m = \sum_{k=1}^{(N_T)_m}\sum_{S\in\mathcal{T}_m}|S|\left(\frac1{d+1}\sum_{K\in\mathcal{M}_S}u_{0,K}^k u_{i,K}^k\right)\nabl\phi(S)\cdot\nabla\psi_{\mathcal{T}_m}(S).
\]
 We remark that, from Lemma \ref{lem:difminmax},
\[
	\left| \frac1{d+1}\sum_{K\in\mathcal{M}_S} u_{0,K}^k u_{i,K}^k- u_{0,S}^ku_{i,S}^k\right| \le |\mathcal{T}|\left(2|\nabl \lunk\sqrt{u_0}\runk_{\!\mathcal{T}_m}(S)|+|\nabl \lunk\sqrt{u_0}u_i\runk_{\!\mathcal{T}_m}(S)|\right).
\]
Therefore, using the Cauchy-Schwarz inequality and \eqref{eq:difminmax}, we get
\[
	|T^{(21)}_m -T^{(2)}_m|^2 \le 3|\mathcal{T}|^2\left(\Vert \nabl\lunk \sqrt{u_0}u_i\runk_{\!\mathcal{T}_m} \Vert_{L^2}^2+2\Vert\nabl \lunk \sqrt{u_0}\runk_{\!\mathcal{T}_m} \Vert_{L^2}^2\right) \Vert\nabl\lunk \phi\runk_{\!\mathcal{T}_m}\Vert_{L^2}^2 \Vert\nabl\psi\Vert_\infty^2.
\]
Therefore, using \eqref{eqn:defC_phi}, we conclude that $\lim_{m\to +\infty}  (T^{(21)}_m -T^{(2)}_m) =0$, which yields
\[
 \lim_{m\to +\infty}  T^{(2)}_m =  \int_0^T\int_\Omega u_{0}  u_{i} \nabl \phi\cdot \nabl \psi {\rm d}x{\rm d}t.
\]

\end{proof}

\section{Numerical experiments}\label{sec:num}
All our implementations are done in {\tt python} and C languages, using  {\tt gmsh} and {\tt fenicsx} to generate the mesh.
Although the numerical analysis is done in terms of entropy variables, numerical tests show that the use of the ion concentrations, electric potential and solvent concentration as primary unknowns together with Newton's method leads to much better numerical properties for getting an approximate solution to the control volume finite element scheme.
\subsection*{Test 1}
We conduct the numerical experiment done in \cite{cances2023} in two space dimensions instead of one space dimension.
We consider $\Omega = (0,1)\times(0,0.1)$ and two ion species $u_1$ and $u_2$ with initial conditions
\begin{align*}
	u_1^0(x,y) = 0.2+0.1(x-1) \quad u_2^0(x,y) =0.4\quad\text{for all }(x,y)\in\Omega
\end{align*}
and charges $z_1=2$ and $z_2=1$, $\beta =1$ and $\lambda^2 = 10^{-2}$.
Furthermore, we apply no-flux boundary conditions everywhere for the ion species and Dirichlet boundary conditions for the electric potential on $\Gamma_D = \{0,1\}\times(0,0.1)$ with
\begin{align*}
	\phi^D(t,0,y) = 10\quad\phi^D(t,1,y)=0\quad\text{for }t\in(0,T)\quad\text{and }y\in(0,0.1).
\end{align*}
We set the time step size to $\tau=5\times 10^{-3}$ and let $T=1$. 
The initial mesh for this test is depicted in Figure \ref{fig:mesh_test1}.
All subsequent meshes are obtained by dividing every triangle into four subtriangles by connecting the midpoints of every edge.
A reference solution is computed with $|\mathcal{T}| = 2\times 10^{-3}$ and $\tau=5\times 10^{-3}$ and $89472$ triangles and $44737$ vertices.
In Figure \ref{fig:conv_test1} we see the $L^\infty_t(L^2_x)$ errors with respect to the reference solution.
We observe here an experimental convergence rate of one for the scheme using the maximum values for the quantities ${u}_{i,S}^k$ (see \eqref{eq:uztmax}) for every $i=0,\dots,n$ and simplex $S$.
In contrast we get an experimental convergence rate of two for the scheme using the arithmetic mean \eqref{eq:uztave}.\\
In Figure \ref{fig:concentrations_test1} is the solution plotted at $T=1$ and $y=0$.
Observe that the solvent concentration is almost zero for $x\in[0.8,1]$ and the scheme performs well in the degenerate regime.
\begin{figure}[!ht]
	\centering
\begin{tikzpicture}
\begin{loglogaxis}[
    width=10cm,
    height=7cm,
    grid=both,
    xlabel={Mesh size},
    ylabel={Error in $L^\infty_t(L^2_x)$},
    legend style={
        at={(1.0,0.0)},
        anchor=south east,
        draw=none,
        fill=none
    },
    xmin=6.337e-3, xmax=6.337e-2,
    ymin=1e-4, ymax=4e-2,
]

\addplot+[mark=o] coordinates {
		(0.0633706715596502,0.0139448861361432)
		(0.0316853357798251,0.00387556111988961)
		(0.0158426678899126,0.000952187858187658)
		(0.00792133394495633,0.00019366905783062)
};
\addlegendentry{Error (mean value)}
\addplot+[mark=o] coordinates {
		(0.0633706715596502, 0.0348378919095286)
		(0.0316853357798251, 0.0170341279274872)
		(0.0158426678899126, 0.00820645252469056)
		(0.00792133394495633,0.00398937044494492)
};
\addlegendentry{Error (max value)}

\addplot[domain=1e-3:1e0, samples=100, dotted] {x^2};
\addlegendentry{$|\mathcal{T}|^2$}
\addplot[domain=1e-3:1e0, samples=100, dotted] {x};
\addlegendentry{$|\mathcal{T}|$}

\end{loglogaxis}
\end{tikzpicture}
\caption{Relative error under space grid refinement using the arithmetic mean}	
\label{fig:conv_test1}
\end{figure}
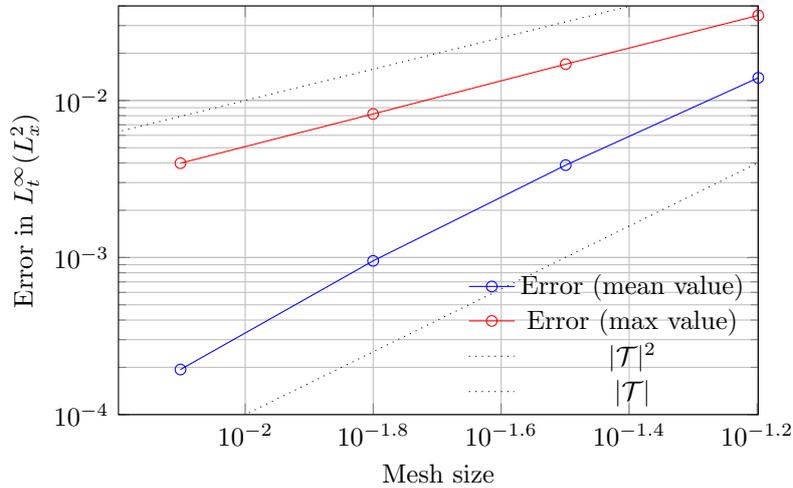
\begin{figure}[!ht]
	\begin{subfigure}{0.5\textwidth}
		\includegraphics[width=\textwidth]{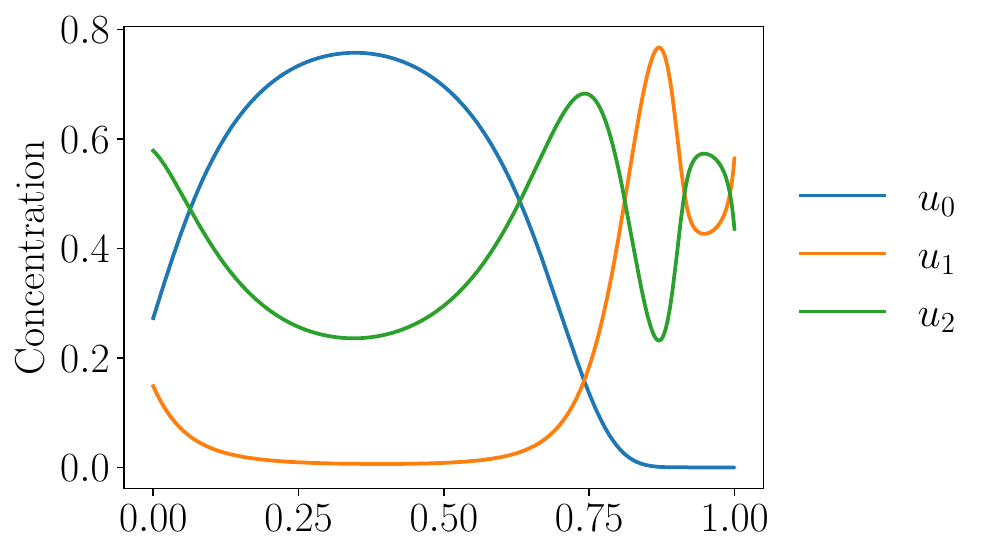}
	\end{subfigure}
	\begin{subfigure}{0.5\textwidth}
		\includegraphics[width=\textwidth]{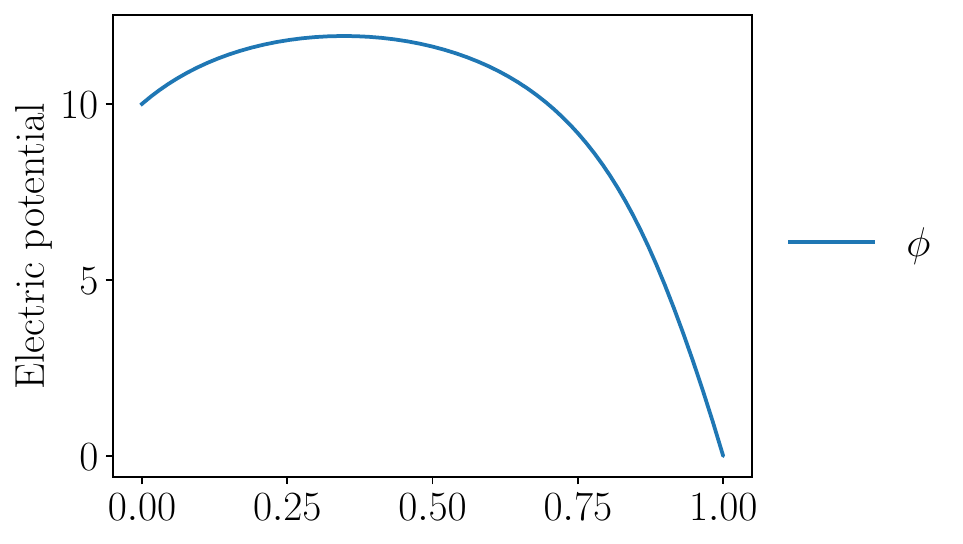}
	\end{subfigure}
	\caption{Concentrations of ion species, solvent and electric potential at time $T=1$ and $y=0$}
	\label{fig:concentrations_test1}
\end{figure}

\subsection*{Test 2}
We consider the calcium-selection ion channel in three space dimensions.
Numerical simulations in one space dimension can be found in \cite{burger2012} and in two space dimensions in \cite{cances2019}.
Immobile oxygen ions are placed in the channel.
These contribute to the electric potential via $f(x,y,z) = -\frac12 u_{\text{ox}}$ with the concentration $u_{\text{ox}}$ given by
\begin{align*}
	u_{\text{ox}}(x,y,z) := 0.84
	\begin{cases}
		10(x-0.35)&\text{if } 0.35\leq x\leq 0.45,\\
		1&\text{if }0.45\leq x\leq 0.55,\\
		10(0.65-x)&\text{if } 0.55\leq x\leq 0.65,\\
		0&\text{otherwise}.
	\end{cases}
\end{align*}
All constants can be found in \cite[Table 1]{burger2012}.
As initial conditions we used affine-linear functions that satisfy the Dirichlet boundary conditions.
In this test, we use the average value formulation of the scheme defined by \eqref{eq:uztave}.
The initial mesh with 866 tetrahedra and 195 nodes before refinement is depicted in Figure \ref{fig:mesh_test2}.
All subsequent meshes are obtained by dividing every simplex into 8 sub-simplices by connecting the mid points of every edge.
A reference solution is obtained with $\tau=5\times10^{-3}$ and a mesh of $\Omega$ with 1982464 tetrahedra, 355145 nodes, which corresponds to $|\mathcal{T}|\approx 0.025$.
In Figure \ref{fig:conv_test2} we see that the error converges in the $L^\infty(0,T;L^2(\Omega))$-norm with approximately order $1.5$.
This is to be expected since $u_{\text{ox}}\in H^{s}(\Omega)$ for every $s<\frac32$ and $1-\sum_{i=0}^n u_i = u_\text{ox}$.
Therefore we cannot expect second order convergence here and observe a reduced order.\\
In Figure \ref{fig:concentrations_test2} we observe the solution at time $T=1$ and $z=0.5$.
We also see here that the solvent concentration vanishes in the channel and still the scheme behaves well in this region.
\begin{figure}[!ht]
	\centering
	\begin{subfigure}{0.5\textwidth}
		\includegraphics[width=1\textwidth]{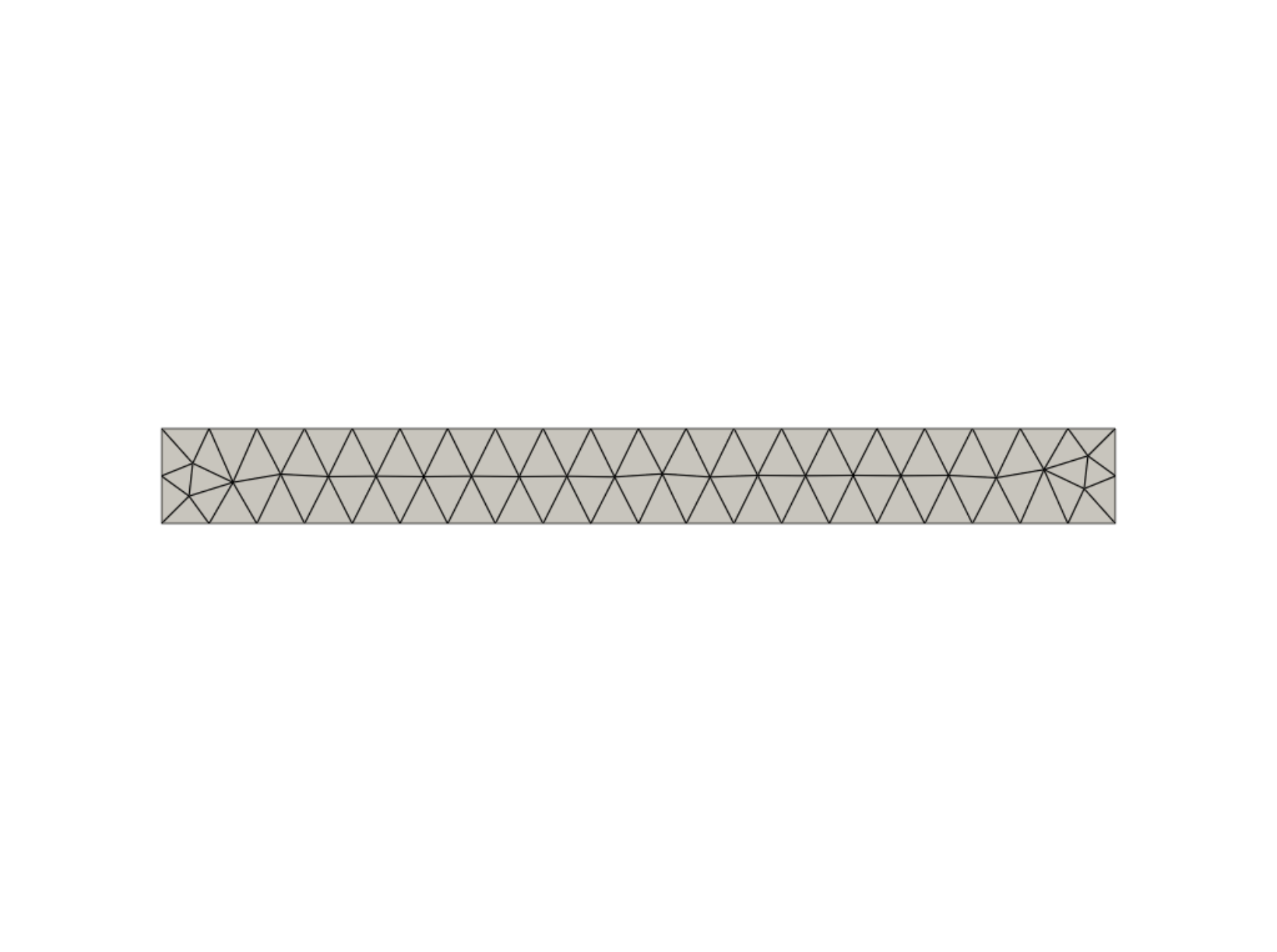}
		\caption{Initial mesh for Test 1}
		\label{fig:mesh_test1}
	\end{subfigure}%
	\begin{subfigure}{0.5\textwidth}
		\includegraphics[width=1\textwidth]{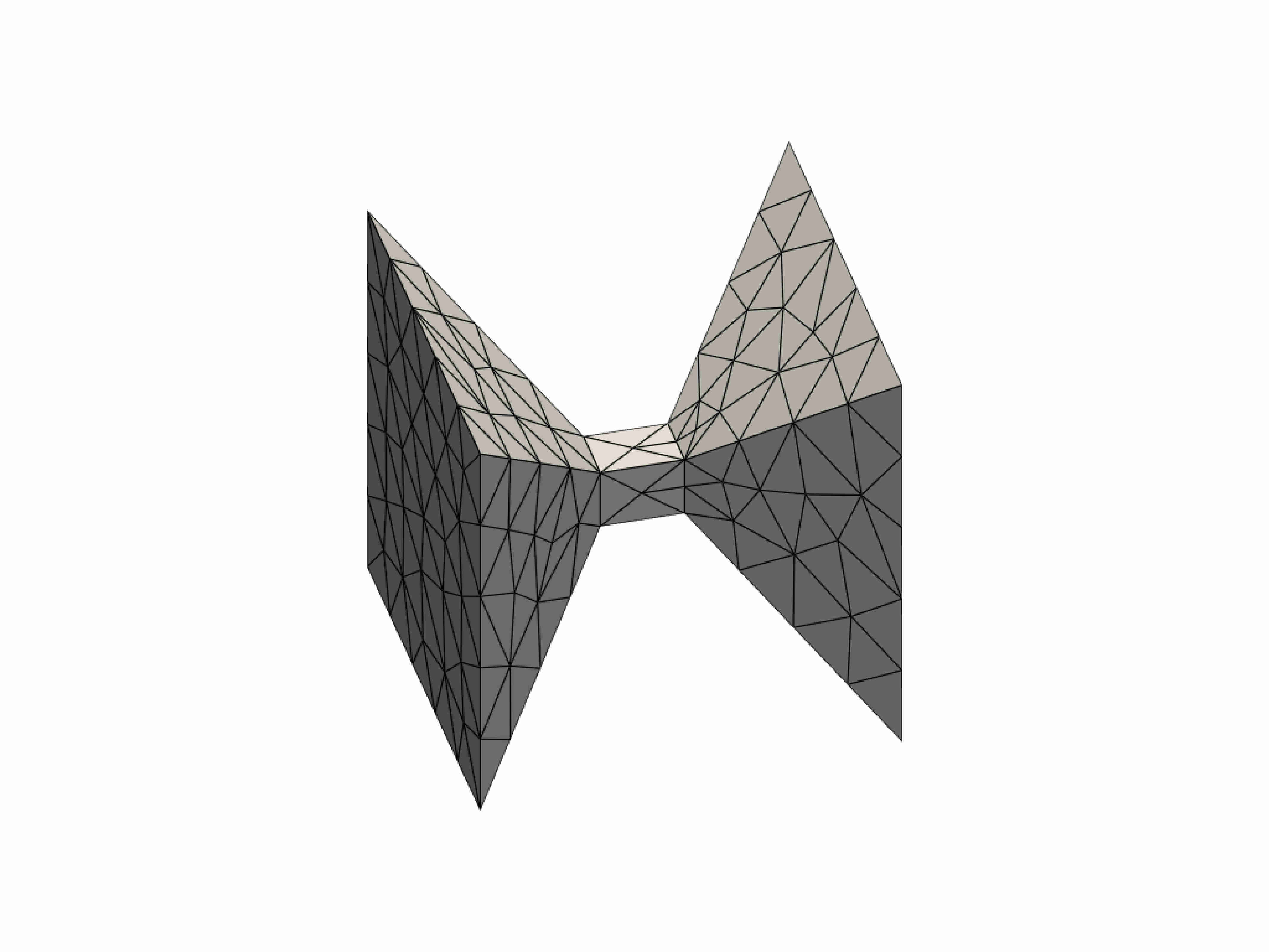}
		\caption{Initial mesh for Test 2}
		\label{fig:mesh_test2}
	\end{subfigure}
	\caption{Initial meshes used in numerical Experiments}
\end{figure}

\begin{figure}[!ht]
	\centering
\begin{tikzpicture}
\begin{loglogaxis}[
    width=10cm,
    height=7cm,
    grid=both,
    xlabel={Mesh size $|\mathcal{T}|$},
    ylabel={Error in $L^\infty_t(L^2_x)$},
    legend style={
        at={(1.0,0.0)},
        anchor=south east,
        draw=none,
        fill=none
    },
    xmin=4e-2, xmax=3.6e-1,
    ymin=1e-3, ymax=4e-2,
]

\addplot+[mark=o] coordinates {
	(0.36,   0.0205089227404318) 
    (0.2	,0.0127469983596375 )
    (0.1	,0.0052417665198347)
    (0.05	,0.00179834702418443)
};
\addlegendentry{Error}

\addplot[domain=1e-3:1e0, samples=100, dotted] {0.4*x^(1.5)};
\addlegendentry{$|\mathcal{T}|^{1.5}$}

\end{loglogaxis}
\end{tikzpicture}
\caption{Relative error under space grid refinement using the arithmetic mean}	
\label{fig:conv_test2}
\end{figure}
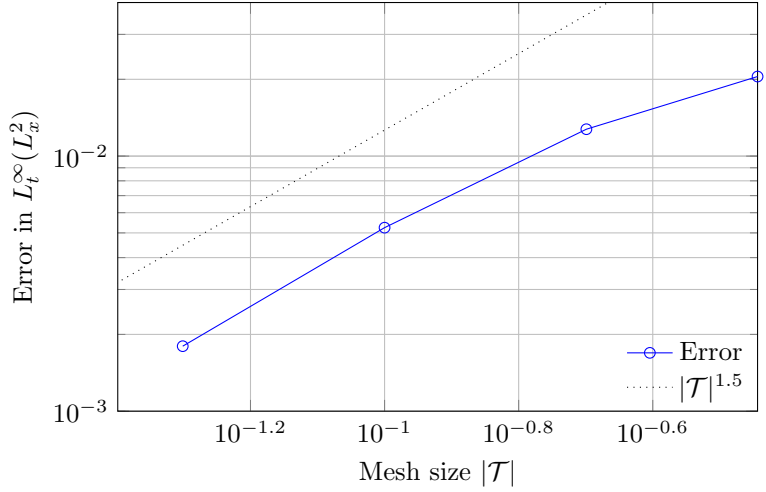
\begin{figure}[!ht]
	\centering
	\begin{subfigure}{0.5\textwidth}
		\includegraphics[width=\linewidth]{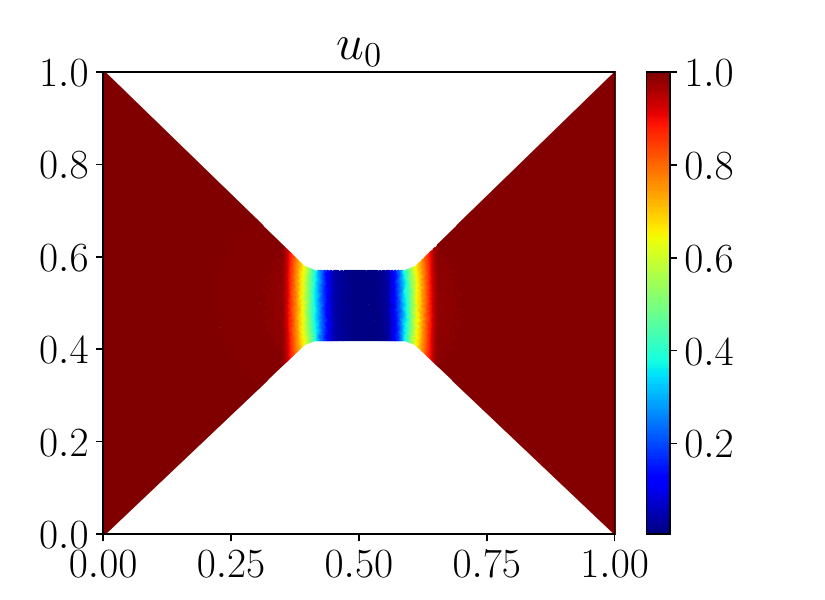}
	\end{subfigure}%
	\begin{subfigure}{0.5\textwidth}
		\includegraphics[width=\linewidth]{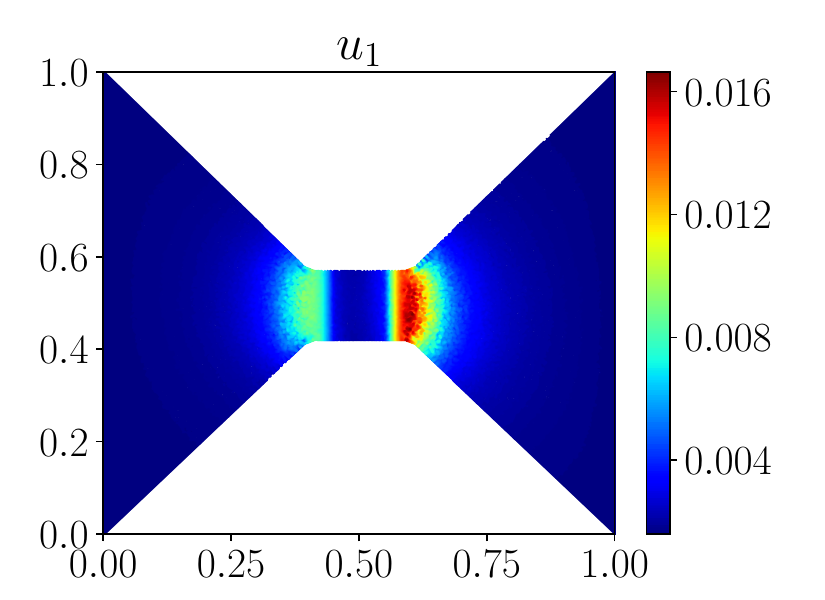}
	\end{subfigure}
	\begin{subfigure}{0.5\textwidth}
		\includegraphics[width=\linewidth]{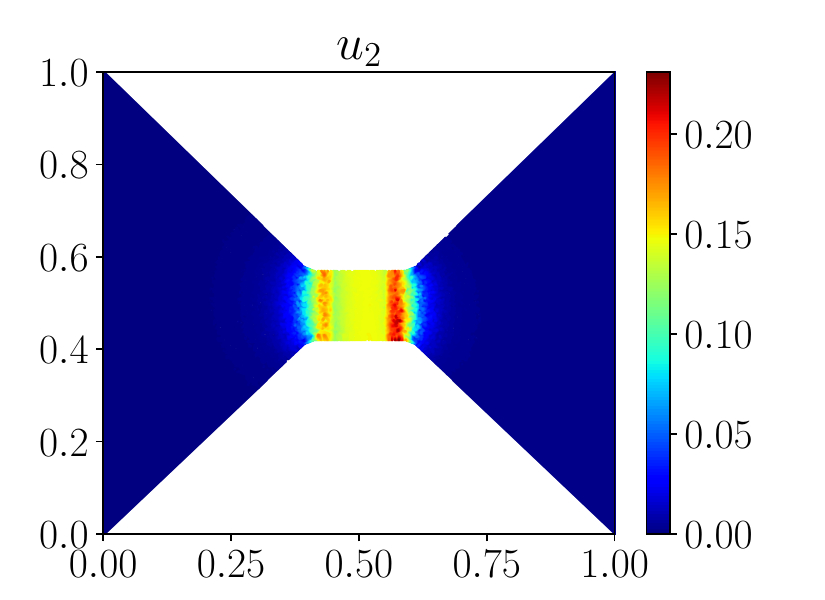}
	\end{subfigure}%
	\begin{subfigure}{0.5\textwidth}
		\includegraphics[width=\linewidth]{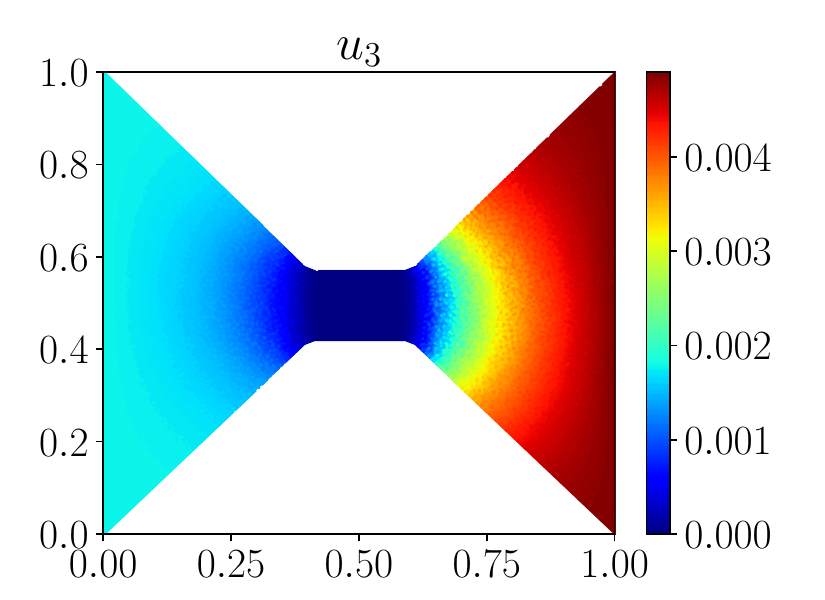}
	\end{subfigure}
	\begin{subfigure}{0.5\textwidth}
		\includegraphics[width=\linewidth]{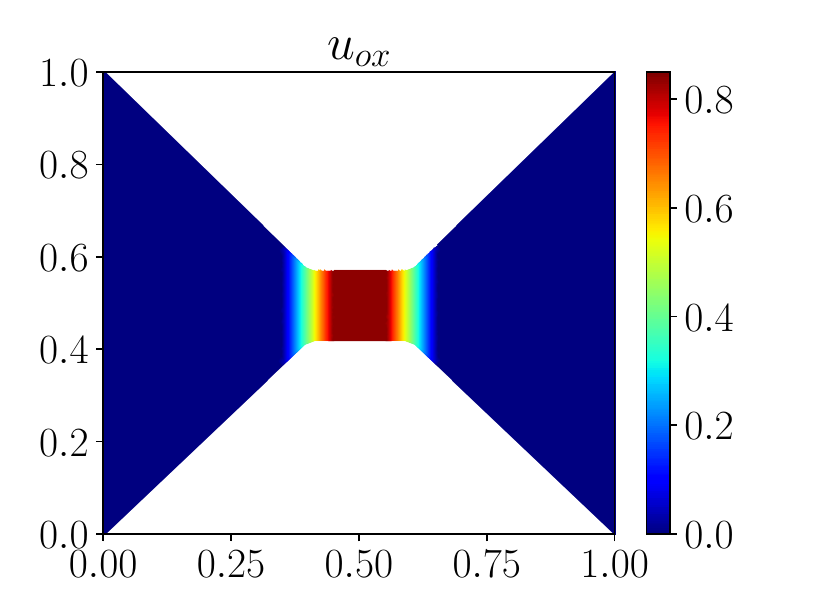}
	\end{subfigure}%
	\begin{subfigure}{0.5\textwidth}
		\includegraphics[width=\linewidth]{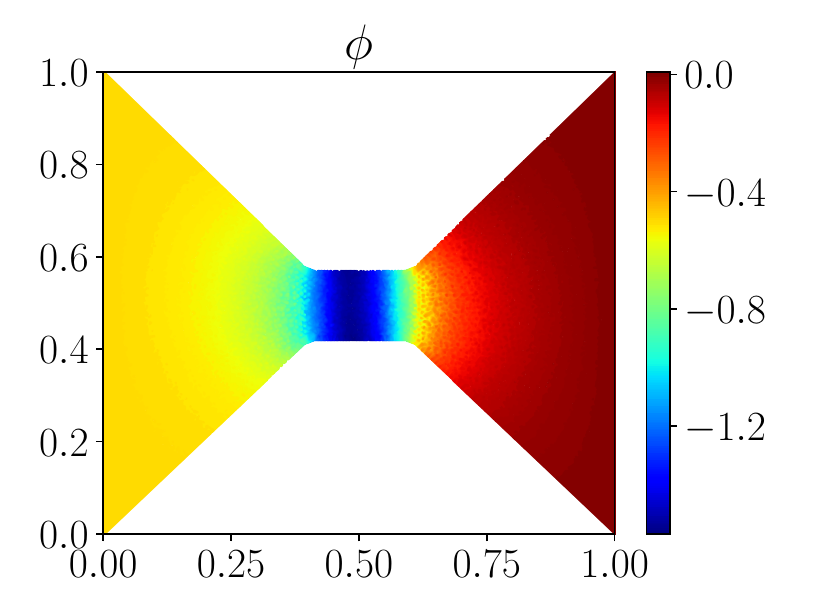}
	\end{subfigure}
	\caption{Concentrations and electric potential at time $T=1$ and $z=0.5$}
	\label{fig:concentrations_test2}
\end{figure}

\appendix

\section{Technical lemmas}\label{sec:teclem}

The next lemma, proved in \cite{cancesguichard2016cvfe,eymmal2021cv}, plays an essential role in the estimates.
It relates the $L^2$-norms and $H^1$-norms of the discrete spaces $X_{\!\mathcal{M}}$ and $V_{\!\mathcal{T}}$.
\begin{lemma}\label{lemma:equivalenz_of_H1L2norms}
	 There are constants $C_{\min},C_{\max}$ only depending on $d$ and $\theta_{\!\mathcal{T}}$, the regularity factor of the mesh, such that, for any $u:=(u_K)_{K\in\mathcal{M}}$,
	\begin{align}\label{eq:gradt}
		C_{\min}\frac{|S|}{h_S^2}\sum_{\{K,L\}\in \mathcal{E}_S}(u_K-u_L)^2\leq|S|(|\nabl u_{\!\mathcal{T}}(S)|)^2\leq C_{\max}\frac{|S|}{h_S^2}\sum_{\{K,L\}\in\mathcal{E}_S}(u_K-u_L)^2\quad\forall S\in\mathcal{T}.
	\end{align}
	and
	\begin{align}\label{eq:equivmt}
		\|u_{\!\mathcal{T}}\|_{L^2(\Omega)}\leq \|u_{\!\mathcal{M}}\|_{L^2(\Omega)}\hbox{ and } \|u_{\!\mathcal{M}} - u_{\!\mathcal{T}}\|_{L^2(\Omega)}\le |\mathcal{T}| \|\nabl u_{\!\mathcal{T}}\|_{L^2(\Omega)^d}.
	\end{align}
\end{lemma}

The next lemma provides algebraic relations, that are used to derive the entropy estimate in Lemma \ref{lem:entropy_estimate}.
\begin{lemma}\label{lemma:useful_inequalities}
	Let $0 < a_i,a'_i\leq 1$ for $i=1,\ldots,n$ and  $0 < b,b'\leq 1$ be given such that it holds
\begin{equation}\label{eq:sumone}
 \sum_{i=1}^n a_i^2 + b^2 = 1\hbox{ and } \sum_{i=1}^n {a'_i}^2 + {b'}^2 = 1.
\end{equation}
	Then
	\begin{subequations}
	\begin{equation}\label{eq:(i)}
				(b' a_i - {b} {a'_i})^2
				\leq \max({{a_i}^2},{{a'_i}^2})\max({b^2},{{b'}^2})\left(\log\left(\frac{{{a_i}^2}}{{b^2}}\right)-\log\left(\frac{{{a'_i}^2}}{{{b'}^2}}\right)\right)^2,
			\end{equation}
			and
			\begin{multline}\label{eq:(ii)}
				\sum_{i=1}^n (b a_i^2 - {b'} {a'_i}^2)^2 + ({b'} -b)^2+ ({b'}^2 -b^2)^2\\
				\leq 12 \sum_{i=1}^n\max({{a_i}^2},{{a'_i}^2})\max({b^2},{{b'}^2})\left(\log\left(\frac{{{a_i}^2}}{{b^2}}\right)-\log\left(\frac{{{a'_i}^2}}{{{b'}^2}}\right)\right)^2.
			\end{multline}
			\label{eq:(i)and(ii)}
	\end{subequations}
\end{lemma}
\begin{proof}
			We observe that
			\begin{multline*}
				\left(\frac{a_i}{b}-\frac{a'_i}{b'}\right)^2
				= \left(\exp\left(\frac {1} {2} \log\left(\frac {{{a_i}^2}} {b^2}\right)\right)
				-\exp\left(\frac {1} {2} \log\left(\frac{{{a'_i}^2}}{{{b'}^2}}\right)\right)\right)^2\\
				\leq \max\left(\frac{{{a_i}^2}}{{b^2}},\frac{{{a'_i}^2}}{{{b'}^2}}\right)\left(\log\left(\frac{{{a_i}^2}}{{b^2}}\right)-\log\left(\frac{{{a'_i}^2}}{{{b'}^2}}\right)\right)^2.
			\end{multline*}
			Multlipling both sides by ${b^2}{{b'}^2}$ yields
			\begin{multline}\label{eq:majzero}
				\left({{a_i}} {{b'}}-{{a'_i}} {b}\right)^2
				\leq \max({{a_i}^2} {{b'}^2},{{a'_i}^2} {b^2})\left(\log\left(\frac{{{a_i}^2}}{{b^2}}\right)-\log\left(\frac{{{a'_i}^2}}{{{b'}^2}}\right)\right)^2\\
				\leq \max({{a_i}^2},{{a'_i}^2})\max({b^2},{{b'}^2})\left(\log\left(\frac{{{a_i}^2}}{{b^2}}\right)-\log\left(\frac{{{a'_i}^2}}{{{b'}^2}}\right)\right)^2.
	\end{multline}
	This proves (i).

	To prove (ii) we then observe that
			\begin{multline}\label{eq:majone}
 4\sum_{i=1}^n(a_i {b'} - {a'_i} b)^2 = \sum_{i=1}^n\big((a_i-{a'_i})(b+ {b'})+(a_i+{a'_i})({b'} -b)\big)^2\\
 = \sum_{i=1}^n(a_i-{a'_i})^2(b+ {b'})^2 + 2 \sum_{i=1}^n(a_i-{a'_i})(b+ {b'})(a_i+{a'_i})({b'} -b)+ \sum_{i=1}^n(a_i+{a'_i})^2({b'} -b)^2 \\
 =\sum_{i=1}^n(a_i-{a'_i})^2(b+ {b'})^2 + \sum_{i=1}^n(a_i+{a'_i})^2({b'} -b)^2 + 2\sum_{i=1}^n(a_i^2-{a'_i}^2)({b'}^2 -b^2).
\end{multline}
We notice that, owing to \eqref{eq:sumone},
\[
 \sum_{i=1}^n(a_i^2-{a'_i}^2) =  {b'}^2 -b^2,
\]
and
\[
 \sum_{i=1}^n(a_i+{a'_i})^2 \ge \sum_{i=1}^n(a_i^2+{a'_i}^2)  = 2 - ({b'}^2 +b^2) \ge 2 - ({b'} +b)^2,
\]
which provides
\[
 \sum_{i=1}^n(a_i+{a'_i})^2({b'} -b)^2 \ge  2({b'} -b)^2 - ({b'} +b)^2({b'} -b)^2 = 2({b'} -b)^2 - ({b'}^2 -b^2)^2,
\]
Hence we get from \eqref{eq:majone}
\begin{multline}\label{eq:majtwo}
 4\sum_{i=1}^n(a_i {b'} - {a'_i} b)^2 = \sum_{i=1}^n\big((a_i-{a'_i})(b+ {b'})+(a_i+{a'_i})({b'} -b)\big)^2 \\
 \ge \sum_{i=1}^n(a_i-{a'_i})^2(b+ {b'})^2 + 2({b'} -b)^2+ ({b'}^2 -b^2)^2.
\end{multline}

We then notice that
\begin{multline}\label{eq:majthree}
4(b a_i^2 - {b'} {a'_i}^2)^2 = \big( (b-{b'})(a_i^2 + {a'_i}^2) + (b+{b'})(a_i^2 - {a'_i}^2) \big)^2 \\
\le 2 (b-{b'})^2(a_i^2 + {a'_i}^2)^2 + 2 (b+{b'})^2(a_i^2 - {a'_i}^2)^2.
\end{multline}
Using $0\le a_i\le 1$ and $0\le a'_i\le 1$, we have
\[
 (a_i^2 + {a'_i}^2)^2\le 4
\]
and
\[
 (a_i^2 - {a'_i}^2)^2 = (a_i - {a'_i})^2(a_i + {a'_i})^2 \le 4(a_i - {a'_i})^2.
\]
We can then write from  \eqref{eq:majthree}
\[
(b a_i^2 - {b'} {a'_i}^2)^2
\le 2 (b-{b'})^2 + 2 (b+{b'})^2(a_i - {a'_i})^2.
\]
Using \eqref{eq:majtwo}, we get

\[
\sum_{i=1}^n (b a_i^2 - {b'} {a'_i}^2)^2 \le 8 \sum_{i=1}^n(a_i {b'} - {a'_i} b)^2.
\]
This relation, in addition to \eqref{eq:majzero} and \eqref{eq:majtwo}, conclude the proof of the lemma.

\end{proof}

The next two lemmas are essential to perform the switch from entropy variables to the ion species in Section \ref{sec:convscheme}.
\begin{lemma}\label{lemma:useful_inequalities_cvone}
	Let $a,b,a',b'$ be strictly positive real values. We denote by $I(\alpha,\beta) = [\min(\alpha,\beta),\max(\alpha,\beta)]$.
	Then there exist $\widetilde{a}\in I(a,a')$ and  $\widetilde{b}\in I(b,b')$ such that
\begin{equation}\label{eq:widetildeab}
	\widetilde{a} \widetilde{b} \left(\log\left(\frac {a}{b}\right) - \log\left(\frac {a'}{b'}\right)\right) = ab' - a' b.
\end{equation}
\end{lemma}
\begin{proof}
 Let us assume that $\frac {a}{b} \ge \frac {a'}{b'}$. We have the existence of $\widetilde{x} \in [\frac {a'}{b'},\frac {a}{b}]$ such that
\[
 \widetilde{x} \left(\log\left(\frac {a}{b}\right) - \log\left(\frac {a'}{b'}\right)\right) = \frac {a}{b} - \frac {a'}{b'},
\]
which provides
\[
\widetilde{x} b b' \left(\log\left(\frac {a}{b}\right) - \log\left(\frac {a'}{b'}\right)\right) = ab' - a' b.
\]
Since we have
\[
 a' b\le \widetilde{x} b b'\le  ab',
\]
letting, for any $\theta\in [0,1]$, $f(\theta) = (a+\theta(a'-a))(b'+\theta(b-b'))$, we get
\[
 f(1)\le \widetilde{x} b b'\le f(0).
\]
Hence there exists $\theta\in [0,1]$ such that $\widetilde{x} b b' = f(\theta)$. Therefore, setting
\[
 \widetilde{a} = a+\theta(a'-a)\hbox{ and }\widetilde{b} =b'+\theta(b-b'),
\]
we get
\[
 \widetilde{a} \widetilde{b} \left(\log\left(\frac {a}{b}\right) - \log\left(\frac {a'}{b'}\right)\right) = ab' - a' b,
\]
with $\widetilde{a}\in I(a,a')$ and  $\widetilde{b}\in I(b,b')$.

Exchanging the roles of $a$ and $a'$, $b$ and $b'$, we get that this relation also holds if  $\frac {a}{b} \le \frac {a'}{b'}$.

\end{proof}

\begin{lemma}\label{lemma:useful_inequalities_cvtwo}
	Let $b,b'$ be strictly positive real values. We again denote by $I(b,b') = [\min(b,b'),\max(b,b')]$.
	Then there exist $\widehat{b}\in I(b,b')$ and  $\check{b}\in I(b,b')$ such that, for all positive values $a,a'$, the following holds:
\begin{equation}\label{eq:hatcheckb}
	ab' - a'b = \sqrt{\widehat{b}}(\sqrt{b} a - \sqrt{b'} a') - 3 \frac {a+a'} 2 (\sqrt{b} - \sqrt{b'}) \sqrt{\check{b}}
\end{equation}
\end{lemma}
\begin{proof}
We have the relation
\[
 ab' - a'b = \sqrt{\widehat{b}}(\sqrt{b} a - \sqrt{b'} a') - 3 \frac {a+a'} 2 (\sqrt{b} - \sqrt{b'}) \sqrt{\check{b}},
\]
with
\begin{equation}\label{eq:formcv}
 \sqrt{\widehat{b}} = \frac {b+b'}{\sqrt{b} + \sqrt{b'}}\hbox{ and } \sqrt{\check{b}} = \frac 2 3\frac {b+\sqrt{b}\sqrt{b'}+b'}{\sqrt{b} + \sqrt{b'}}.
\end{equation}
Since the expressions of $\widehat{b}$ and $\check{b}$ are symmetric with respect to $b$ and $b'$, we can assume that $b\le b'$. Observing that
\begin{align}\label{eqn:hat_estimate}
 \sqrt{\widehat{b}} - \sqrt{b} =  \frac {\sqrt{b'}}{\sqrt{b} + \sqrt{b'}}(\sqrt{b'} - \sqrt{b})\leq (\sqrt{b'} - \sqrt{b})
\end{align}
and
\begin{align}\label{eqn:check_estimate}
 \sqrt{\check{b}}  - \sqrt{b} = \frac 1 3\frac {\sqrt{b}+ 2\sqrt{b'}}{\sqrt{b} + \sqrt{b'}}(\sqrt{b'} - \sqrt{b})\leq\frac23(\sqrt{b'} - \sqrt{b}),
\end{align}
we get that $\sqrt{\widehat{b}}\in I(\sqrt{b},\sqrt{b'})$ and  $\sqrt{\check{b}}\in I(\sqrt{b},\sqrt{b'})$.
\end{proof}

The next lemma is adapted from \cite{egh2021deg}. It is used to estimate the time translate estimates in Section \ref{sec:anascheme}.

\begin{lemma} 
\label{lem:esttimtran}~ Under the notations of Section \ref{sec:scheme}, let  $(u_K^{k},v_K^{k})_{K,k}$ be given values for $k\in\mathcal{M}$ and $k=0,\ldots,N_T$. We define $\overline{\partial}_t v\in X_{\!\mathcal{M},\tau}$ by the value $\frac {v_K^{k} - v_K^{k-1}}{\tau^k}$ at $(t^k,x)$ for a.e. $x\in K$, all $K\in \mathcal{M}$ and $k=1,\ldots,N_T$.
Then, for any $s\in[0,T]$ and $\zeta\in[0,s]$, the following holds:
\begin{multline}
\int_0^{T-s} \int_\Omega
   (v_{\!\mathcal{M}}(x,t+s) - v_{\!\mathcal{M}}(x,t))\  u_{\!\mathcal{M}}(x,t+\zeta) {\rm d} x{\rm d} t \\
\le 
s\
\Vert  \overline{\partial}_t v_{\!\mathcal{M}} \Vert_{L^2(0,T;H^{-1}_{\!\mathcal{M}}(\Omega))}
\ 
\Vert  \nabl u_{\!\mathcal{T}}     \Vert_{L^2(0,T;L^2(\Omega)^d)}.
\label{eq:esttt2}
\end{multline}
\end{lemma}

\begin{proof}
	Let us define 
	$$
	A(t) :=  \int_\Omega   u_{\!\mathcal{M}}(x,t+\zeta) (v_{\!\mathcal{M}}(x,t+s) - v_{\!\mathcal{M}}(x,t)){\rm d} x.
	$$
	Let $t\in (0,T-s)$. Denoting $\underline{k}(t) =0,\ldots,{N_T}-1$ such that $t^{\underline{k}(t)} \le t <t^{\underline{k}(t)+1}$, we may write

	\[
		A(t) =  \int_\Omega     u_{\!\mathcal{M}}^{\underline{k}(t+\zeta)+1}(x)
		\left(\sum_{k = \underline{k}(t)+1}^{\underline{k}(t+s)}\tau^{k+1} \frac {v_{\!\mathcal{M}}^{k+1}(x) - v_{\!\mathcal{M}}^{k}(x)}{\tau^{k+1}}\right) {\rm d}x,
	\]
	Using the definition of the $\|\cdot\|_{H_{\!\mathcal{M}}^{-1}}$
	\begin{align*}
		A(t)\leq  \left\Vert \nabl u^{\underline{k}(t+\zeta)+1}_{\!\mathcal{T}}\right\Vert_{L^2(\Omega)^d}  \left(\sum_{k=\underline{k}(t)+1}^{\underline{k}(t+s)}\tau^{k+1}\left\Vert \frac {v_{\!\mathcal{M}}^{k+1} - v_{\!\mathcal{M}}^{k}}{\tau^{k+1}}\right\Vert_{H^{-1}_{\!\mathcal{M}}}\right).
	\end{align*}

	Using the Cauchy-Schwartz and Hölder inequalities we obtain
	\begin{align*}
		\int_0^{T-s} A(t){\rm d}t\leq \left(\int_0^{T-s} A_1(t){\rm d}t\right)^{\frac12}\left(\int_0^{T-s} A_2(t){\rm d}t\right)^{\frac12}
	\end{align*}
	with 
	\begin{align*}
		A_1(t):=\sum_{k=\underline{k}(t)+1}^{\underline{k}(t+s)}\tau^{k+1}\left\Vert \nabl u^{\underline{k}(t+\zeta)+1}_{\!\mathcal{T}}\right\Vert_{L^2(\Omega)^d}^2
	\end{align*}
	and
	\begin{align*}
		A_2(t):=\sum_{k=\underline{k}(t)+1}^{\underline{k}(t+s)}\tau^{k+1}\left\Vert \frac{v_{\!\mathcal{M}}^{k+1}-v_{\!\mathcal{M}}^{k}}{\tau^{k+1}}\right\Vert_{H^{-1}_{\!\mathcal{M}}}^2.
	\end{align*}
	To estimate the integral over $A_1(t)$ we use that $\sum_{k=\underline{k}(t)+1}^{\underline{k}(t+s)}\tau^{k+1}=t^{\underline{k}(t+s)}-t^{\underline{k}(t)+1}\leq t+s-t = s$ to obtain
	\begin{align}\label{trt14}
		\int_0^{T-s}\sum_{k=\underline{k}(t)+1}^{\underline{k}(t+s)}\tau^{k+1}\left\Vert \nabl u^{\underline{k}(t+\zeta)+1}_{\!\mathcal{T}}\right\Vert_{L^2(\Omega)^d}^2{\rm d}t\leq s\|\nabl u\|_{L^2(0,T;L^2(\Omega)^d)}^2.
	\end{align}
	We now turn our attention to the integral over $A_2(t)$.
	Denote by $\chi_{k,s}(t)$ the indicator function of $[t_k-s,t_k)$, i.e. $\chi_{k,s}(t)=1$ if $t\in[t^k-s,t^k)$ and $\chi_{k,s}(t)=0$ otherwise.
	Using this we directly obtain
	\begin{align*}
		A_2(t) \le \sum_{k=1}^{{N_T}-1}\tau^{k+1}\chi_{k,s}(t)\left\Vert \frac{v_{\!\mathcal{M}}^{k+1}-v_{\!\mathcal{M}}^{k}}{\tau^{k+1}}\right\Vert_{H^{-1}_{\!\mathcal{M}}}^2.
	\end{align*}
	Integrating $A_2(t)$ now yields
	\begin{align}\label{trt16}
		\int_0^{T-s} A_2(t){\rm d}t\leq\int_{\max\{0,t^k-s\}}^{t^k}\sum_{k=1}^{{N_T}-1}\tau^{k+1}\left\Vert \frac{v_{\!\mathcal{M}}^{k+1}-v_{\!\mathcal{M}}^{k}}{\tau^{k+1}}\right\Vert_{H^{-1}_{\!\mathcal{M}}}^2{\rm d}t
		\leq s \Vert  \overline{\partial}_t v_{\!\mathcal{M}} \Vert_{L^2(0,T;H^{-1}_{\!\mathcal{M}}(\Omega))}.
	\end{align}
	Combining the estimates \eqref{trt14} and \eqref{trt16} leads to \eqref{eq:esttt2}.
\end{proof}

\section{Proof of the existence of a solution to the scheme}\label{sec:existcheme}
We prove the existence of a solution to Scheme \eqref{eq:scheme} by induction over $k$, using a topological degree argument.
For the induction step, we assume that there exists a non-empty solution $(\mu_{i,K}^{k-1},\phi_K^{k-1})$ to the scheme, i.e. there are $M_0^{k-1},\dots,M_n^{k-1}$ such that
\begin{align}\label{eqn:initial_conditions_not_empty}
	\begin{split}
		&\sum_{K\in\mathcal{M}}|K|u_{i,K}^{k-1}=M_i^{k-1}>0\quad\forall i=1,\dots,n\quad\text{and}\\
		&\sum_{K\in\mathcal{M}}u_{0,K}^{k-1}=|\Omega|-\sum_{i=1}^nM_i^{k-1}=:M_0^{k-1}>0.
	\end{split}
\end{align}
\begin{subequations}
The following functions are defined, for a given $\gamma\in [0,1]$ and for any set of real values $(\mu_{j,L})_{j,L},(\phi_L)_L$ by
\begin{align}
	&\mathcal{H}_{i,K}((\mu_{j,L})_{j,L},(\phi_L)_L,\gamma)	:=\gamma|K|\frac{u_{i,K}-u_{i,K}^{k-1}}{\tau^{k}}+(1-\gamma)|K|\frac{u_{i,K}-\frac{M_i^{k-1}}{|\Omega|}}{\tau^{k}}
	\nonumber\\ &+\gamma\sum_{S\in\mathcal{T}_K} D_i|S|{u}_{0,S}{u}_{i,S}\left(\nabl \mu_{i,\mathcal{T}}(S)+\beta z_i\nabl \lunk \phi\runk_{\!\mathcal{T}}(S)\right)\cdot\nabl e_K(S)
	\nonumber\\ &+(1-\gamma)\sum_{S\in\mathcal{T}_K}|S|\nabl \mu_{i,\mathcal{T}}(S)\cdot\nabl e_K(S),  \quad K\in\mathcal{M},\label{eqn:homotopyi}\\
		&\mathcal{H}_{\phi,K}((\mu_{i,L})_{i,L},(\phi_L)_L,\gamma)	:=\nonumber\\ &\int_\Omega\Big(\lambda^2\nabl\lunk \phi\runk_{\!\mathcal{T}}(x)\cdot\nabl e_K(x)- \gamma\big(\sum_{i=1}^n z_i \lunk u_i \runk_{\!\mathcal{M}}(x) + f(x)\big)\chi_K(x)\Big){\rm d}x,  \quad K\in\mathcal{M}\setminus\mathcal{M}_D,\nonumber \\
		&\mathcal{H}_{\phi,K}((\mu_{i,L})_{i,L},(\phi_L)_L,\gamma)	:= \phi_K - \gamma \phi^D_{\!\mathcal{T}},  \quad K\in\mathcal{M}_D,\label{eqn:homotopy}
\end{align}
with
\begin{equation}\label{eqn:homotopytwo}
 u_{0,K} = \frac 1 {1+\sum_{i=1}^n \exp(\mu_{i,K})}\hbox{ and }u_{i,K} = u_{0,K}\exp(\mu_{i,K}),
\end{equation}
For both choices \eqref{eq:uztmax} and \eqref{eq:uztave} for ${u}_{0,S}$ and ${u}_{i,S}$, these are continuous functions of the arguments of $\mathcal{H}$ and remember that
\begin{equation}\label{eqn:homotopythree}
 \frac {u_{0,K}} {d+1}\le {u}_{0,S} \le \max_{L\in \mathcal{M}_S}u_{0,L} \hbox{ and } \frac {u_{i,K}} {d+1}\le {u}_{i,S}\le \max_{L\in \mathcal{M}_S}u_{i,L}\hbox{ for all } K\in \mathcal{M}_S.
\end{equation}
The aim of the next lemmas is to prove that, for any  $\gamma\in[0,1]$, there exist real values $((\mu_{i,K})_{i,K}$ and $(\phi_{K})_K$ such that
\begin{align}
	&\mathcal{H}_{i,K}((\mu_{j,L})_{j,L},(\phi_L)_L,\gamma)=0,\quad\forall i=1,\dots,n,\quad K\in\mathcal{M},\label{eq:degtopi}\\
	&\mathcal{H}_{\phi,K}((\mu_{i,L})_{i,L},(\phi_L)_L,\gamma)	 = 0,\quad\forall K\in\mathcal{M}\label{eq:degtopphi}.
\end{align}
\label{eq:degtop}
\end{subequations}
Let us first state a conservation of mass property.
\begin{lemma}\label{lemma:conservation_of_mass}
	Let $\gamma\in[0,1]$. Assume that $((\mu_{i,K})_{i,K},(\phi_{K})_K)$ satisfy \eqref{eq:degtop}.
	Then the mass of the ion species $u_i$ is conserved, i.e. $\sum_{K\in\mathcal{M}}|K|u_{i,K}=\sum_{K\in\mathcal{M}}|K|u_{i,K}^{k-1}=M_i^{k-1}$.
\end{lemma}
\begin{proof}
	We sum equation \eqref{eq:degtopi} over all $K\in\mathcal{M}$ and obtain
	\begin{multline}\label{eq:summed_conservation}
		\gamma\sum_{K\in\mathcal{M}}|K|\frac{u_{i,K}-u_{i,K}^{k-1}}{\tau^{k}}+(1-\gamma)\sum_{K\in\mathcal{M}}|K|\frac{u_{i,K}-\frac{M_i^{k-1}}{|\Omega|}}{\tau^{k}}\\
		 +\sum_{K\in\mathcal{M}}\gamma\sum_{S\in\mathcal{T}_K} D_i|S|{u}_{0,S}{u}_{i,S}\left(\nabl \mu_{i,\mathcal{T}}(S)+\beta z_i\nabl \lunk \phi\runk_{\!\mathcal{T}}(S)\right)\cdot\nabl e_K(S)\\
		+\sum_{K\in\mathcal{M}}(1-\gamma)\sum_{S\in\mathcal{T}_K}|S|\nabl \mu_{i,\mathcal{T}}(S)\cdot\nabl e_K(S)=0
	\end{multline}
	We can use that $\sum_{K\in\mathcal{M}_S} \nabl e_K(S)=0$ for all $S\in\mathcal{T}$.
	With this we obtain from \eqref{eq:summed_conservation} that
	\begin{align*}
		\gamma\sum_{K\in\mathcal{M}}|K|\frac{u_{i,K}-u_{i,K}^{k-1}}{\tau^{k}}+(1-\gamma)\sum_{K\in\mathcal{M}}|K|\frac{u_{i,K}-\frac{M_i}{|\Omega|}}{\tau^{k}}=0
	\end{align*}
	With \eqref{eqn:initial_conditions_not_empty} follows the claim.
\end{proof}
The next lemma establishes a uniform bound on a weighted $H^1$-norm of the gradient of the entropy variables.
\begin{lemma}\label{lemma:entropy_gamma}
	Let $\gamma\in[0,1]$. Let $((\mu_{i,K})_{i,K},(\phi_{K})_K)$ be a solution to \eqref{eq:degtop}.
	Then the following holds:
	\begin{equation}
\label{eq:degtopone}
			C_{\min} \tau^{k}\sum_{i=1}^n\sum_{S\in\mathcal{T}} \frac{|S|}{h_S^2}
			\Big(\gamma \frac {D_i} 2 {u}_{0,S}{u}_{i,S} + (1-\gamma)\Big)\sum_{ \{K,L\}\in\mathcal{E}_S }(\mu_{i,K} - \mu_{i,L})^2\\
			\leq n |\Omega| +\gamma \tau^{k}\cter{cte:phid}
	\end{equation}
	where $\cter{cte:phid}$ is defined in \eqref{eq:entropone}.
\end{lemma}
\begin{proof}
	We observe that the time term in \eqref{eqn:homotopyi} is written under the form $|K|\frac{u_{i,K}-\left(\gamma u_{i,K}^{k-1} + (1-\gamma)\frac{M_i^{k-1}}{|\Omega|}\right)}{\tau^{k}}$.
	Hence, multiplying \eqref{eqn:homotopyi} by $\mu_{i,K}$ and summing over $K\in\mathcal{M}$ and $i=1,\dots,n$ and following the proof of Lemma \ref{lem:entropy_estimate} we obtain
	\begin{multline*}
		\sum_{i=0}^n\sum_{K\in\mathcal{M}}|K|\zeta(u_{i,K}) +\tau^{k} C_{\min}\sum_{i=1}^n\sum_{S\in\mathcal{T}}|S|\left(\gamma\frac {D_i} 2{u}_{0,S}{u}_{i,S}+(1-\gamma)\right)|\nabl \mu_{i,\mathcal{T}}(S)|^2\\
		\leq \sum_{i=0}^n\sum_{K\in\mathcal{M}}|K| \zeta\left(\gamma u_{i,K}^{k-1}+ (1-\gamma)\frac{M_i^{k-1}}{|\Omega|}\right) +\gamma\tau^{k}\cter{cte:phid}.
	\end{multline*}
	Using $\gamma u_{i,K}^{k-1}+ (1-\gamma)\frac{M_i^{k-1}}{|\Omega|} \in (0,1)$ and Lemma \ref{lemma:equivalenz_of_H1L2norms} yields the claim.
\end{proof}
To use a topological degree argument we have to show that the set of solutions to \eqref{eq:degtop} is uniformly bounded.
For this we first show that the solvent concentrations are uniformly bounded from below.
We use here a similar strategy to \cite{cancesguichard2016cvfe}.
\begin{lemma}\label{lemma:boundednessOfu_0}
	Let $\gamma\in[0,1]$. Let $((\mu_{i,K})_{i,K},(\phi_{K})_K)$ be a solution to \eqref{eq:degtop}. Then there exists $\underline{C}>0$, which does neither depend on $\gamma$ nor on $((\mu_{i,K})_{i,K},(\phi_{K})_K)$, such that
	\begin{align}
		u_{0,K}\geq \underline{C} \quad\forall K\in\mathcal{M}.\label{eq:minuzdegtop}
	\end{align}
\end{lemma}
\begin{proof}
	By Lemma \ref{lemma:conservation_of_mass} and \eqref{eqn:initial_conditions_not_empty} we obtain
	\[
		\sum_{K\in \mathcal{M}} |K|u_{0,K} =  M_0^{k-1}>0.
	\]
	Hence there exists $K_0\in \mathcal{M}$ such that $u_{0,K_0}\ge C_0 := \frac {M_0^{k-1}} {|\Omega| } >0$ (otherwise $\sum_{K\in \mathcal{M}} |K|u_{0,K} < M_0^{k-1}$).\\
	Let $L\in\mathcal{M}$.
	Then there exists a sequence $K_0,K_1,\dots,K_M=L$ of distinct elements of $\mathcal{M}$ (therefore $M\le \#\mathcal{M}-1$) and for any $m=1,\dots,M$, there exists $S_m\in\mathcal{T}$ such that $K_{m-1},K_m\in\mathcal{M}_S$.
	Let us prove by induction that there exists $C_m>0$ such that $u_{0,K_m}\ge C_m$ for $m=0,\dots,M$.
	For $m=0$ we have $u_{0,K_0}\geq C_0$.\\
	Let us assume that there exists $1>C_{m-1}>0$ such that $u_{0,K_{m-1}}\ge C_{m-1}$.
	We therefore deduce from \eqref{eqn:homotopythree} that ${u}_{0,S_m} \geq \frac {C_{m-1}} {d+1}$.
	We consider the two cases $u_{0,K_{m}}\ge \frac 1 2$ or $u_{0,K_{m}} < \frac 1 2$.\\
	Let us assume that  $u_{0,K_{m}} < \frac 1 2$. Since $\sum_{i=1}^n u_{i,K_{m}} + u_{0,K_{m}} = 1$, we deduce that there exists $i_0\in\{1,\ldots,n\}$ such that
	\begin{align}\label{eqn:u_i_0_lower}
		u_{i_0,K_{m}} \ge \frac {1- u_{0,K_{m}}} {n} \ge \frac {1} {2n}.
	\end{align}
	Hence, ${u}_{i_0,S_m} \geq \frac {1} {2n}$.\\
	With \eqref{eq:degtopone} in Lemma \ref{lemma:entropy_gamma} we deduce that there exists $\widehat{C}>0$ which does neither depend on $\gamma$ nor on $((\mu_{i,K})_{i,K},(\phi_{K})_K)$, such that
	\[
		\Big(\gamma {u}_{0,S}{u}_{i_0,S} + (1-\gamma)\Big)(\mu_{i_0,K_m} - \mu_{i_0,K_{m-1}})^2 \le \widehat{C}.
	\]
	Using the definition \eqref{eqn:homotopytwo} yields
	\[
		\Big(\gamma {u}_{0,S}{u}_{i_0,S} + (1-\gamma)\Big)\left(\log\left(\frac {u_{i_0,K_{m}}}{u_{0,K_{m}}}\right) - \log\left(\frac {u_{i_0,K_{m-1}}}{u_{0,K_{m-1}}}\right)\right)^2 \leq \widehat{C}.
	\]
	This implies
	\[
		\log\left(\frac {u_{i_0,K_{m}}}{u_{0,K_{m}}}\right) \le \log\left(\frac {u_{i_0,K_{m-1}}}{u_{0,K_{m-1}}}\right) + \sqrt{\frac {\widehat{C}} { \gamma {u}_{0,S}{u}_{i_0,S} + (1-\gamma)}}.
	\]
	We notice that \eqref{eqn:u_i_0_lower} and ${u}_{0,S_m}\geq \frac {C_{m-1}} {d+1}$ provide
	\[
		\frac {\widehat{C}} { \gamma {u}_{0,S_m}{u}_{i_0,S_m} + (1-\gamma)} \le \frac {\widehat{C}} { \gamma  \frac {C_{m-1}} {d+1}\frac {1} {2n} + (1-\gamma)} \le \frac {2n\widehat{C}(d+1)} { C_{m-1}}.
	\]
	The induction hypothesis  $u_{0,K_{m-1}}\ge C_{m-1}$ leads to $u_{i_0,K_{m-1}} \le 1-C_{m-1}$ and therefore to
	\[
		\frac {u_{i_0,K_{m-1}}}{u_{0,K_{m-1}}} \leq \frac {1-C_{m-1}}{C_{m-1}}.
	\]
	The relation \eqref{eqn:u_i_0_lower} yields
	\[
		\frac {u_{i_0,K_{m}}}{u_{0,K_{m}}} \geq \frac {1/(2n)}{u_{0,K_{m}}}.
	\]
	Gathering the preceding inequalities, this yields
	\[
		u_{0,K_{m}} \ge \widetilde{C}_m := \frac 1 {2n}\exp\left(\log\left(\frac {C_{m-1}}{1-C_{m-1}}\right)  - \sqrt{ \frac {2n \widehat{C}(d+1)} { C_{m-1}}}\right)>0.
	\]
	Combining with the case $u_{0,K_m}\ge \frac12$ we obtain the lower bound
	\[
		u_{0,K_{m}}\ge C_m := \min\left(\frac 1 2, \widetilde{C}_m  \right).
	\]
	Hence we define the sequence $(C_m)_{m\in\mathbb{N}}$ of strictly positive reals by $C_0= \frac {M_0^{k-1}} {|\Omega| } >0$ and, for all $m\in\mathbb{N}$,
	\[
		C_m = \min\left(\frac 1 2,   \frac 1 {2n}\exp\left(\log\left(\frac {C_{m-1}}{1-C_{m-1}}\right)  - \sqrt{ \frac {2n \widehat{C}(d+1)} { C_{m-1}}}\right)\right).
	\]
	Since every $L\in\mathcal{M}$ can be connected by a sequence of $\#\mathcal{M}-1$ elements of $\mathcal{M}$ in the above sense, it suffices now to define $\underline{C}>0$ by
\begin{align*}
	\underline{C}:=\min_{m=0,\ldots,\#\mathcal{M}-1} C_m>0,
\end{align*}
for concluding the proof of the lemma.
\end{proof}
We now use a similar strategy to prove the uniform boundedness of the entropy variables.
\begin{lemma}\label{lem:entropy_bout_top}
	Let $\gamma\in[0,1]$. Let $((\mu_{i,K})_{i,K},(\phi_{K})_K)$ be a solution to \eqref{eq:degtop}. Then there exists $\overline{C}>0$, which does neither depend on $\gamma$ nor on $((\mu_{i,K})_{i,K},(\phi_{K})_K)$, such that
	\begin{align*}
		|\mu_{i,K}|\leq \overline{C} \quad\forall K\in\mathcal{M},\quad\forall i=1,\ldots,n.
	\end{align*}
\end{lemma}
\begin{proof}
	With Lemma \ref{lemma:conservation_of_mass} we obtain that
	\[
		\sum_{K\in \mathcal{M}} |K|u_{i,K} = M_i^{k-1} >0,
	\]
	and therefore there exists $K_0\in \mathcal{M}$ such that $u_{i,K_0}\ge  \frac {M_i^{k-1}} {|\Omega|} \ge  \frac {\min_i M_i^{k-1}} {|\Omega|}>0$. Since
	\[
		u_{i,K_0} = u_{0,K_0}\exp(\mu_{i,K_0}),
	\]
	we deduce, using \eqref{eq:minuzdegtop}, that
	\[
		\log\left(\frac {\min_i M_i^{k-1}} {|\Omega|}\right)\le \mu_{i,K_0} \le\log\left(\frac {1} {\underline{C}}\right).
	\]
	We then define
	\begin{equation}\label{eq:defcz}
		C_0 := \max\left(\log\left(\frac {1} {\underline{C}}\right), -\log\left(\frac {\min_i M_i^{k-1}} {|\Omega|}\right)\right).
	\end{equation}
	Let $L\in\mathcal{M}$. Then there exists a sequence $K_0,K_1,\dots,K_{M_L}=L$ of distinct elements of $\mathcal{M}$ (therefore we again have $M_L\le \#\mathcal{M}-1$) and for any $m=1,\dots,M_L$, there exists $S_m\in\mathcal{T}$ such that $K_{m-1},K_m\in\mathcal{M}_S$.
	We proceed again by induction.
	For $m=0$ the bound follows from the above.\\
	Let us assume that there exists $C_{m-1}>0$ such that $|\mu_{i,K_{m-1}}|\le C_{m-1}$. Therefore, using \eqref{eqn:homotopythree}, we have ${u}_{i,S_m} \ge\frac 1 {d+1} u_{i,K_{m-1}} \ge \frac 1 {d+1}\underline{C}\exp(-C_{m-1})$. Since ${u}_{0,S_m} \ge \underline{C}$, we get
	${u}_{i,S_m} {u}_{0,S_m}\geq \frac 1 {d+1}\exp(-C_{m-1})\underline{C}^2$.
	With \eqref{eq:degtopone}  in Lemma \ref{lemma:entropy_gamma}, we deduce that there exists $\widehat{C}>0$ which does neither depend on $\gamma$ nor on $((\mu_{i,K})_{i,K},(\phi_{K})_K)$,  such that
	\[
		\Big(\gamma {u}_{0,S_m}{u}_{i,S_m} + (1-\gamma)\Big)(\mu_{i,K_m} - \mu_{i,K_{m-1}})^2 \le \widehat{C}.
	\]
	This implies
	\[
		|\mu_{i,K_m}| \le |\mu_{i,K_{m-1}}| + \sqrt{\frac {\widehat{C}} { \gamma {u}_{0,S_m}{u}_{i,S_m} + (1-\gamma)}} \le  C_m := C_{m-1} + \sqrt{\frac {\widehat{C}} { \frac 1 {d+1}\exp(-C_{m-1})\underline{C}^2}}.
	\]
	Hence we define the sequence $(C_m)_{m\in\mathbb{N}}$ of strictly positive reals given by \eqref{eq:defcz} and, for all $m\in\mathbb{N}$,
	\[
		C_m := C_{m-1} + \sqrt{\frac {\widehat{C}(d+1)} { \exp(-C_{m-1})\underline{C}^2}}.
	\]
	Thus it is again enough to define
	\[
		\overline{C}=\max\{C_{m},m=0,\ldots \mathcal{M}-1\}
	\]
    to conclude the lemma.
\end{proof}
Since \eqref{eqn:homotopyi} is not a linear problem for $\gamma=0$, we have to use \cite[Lemma 1.2.2]{dinca2021} to obtain that $\operatorname{deg}(\mathcal{H})\neq 0$ for $\gamma=0$.
To apply \cite[Lemma 1.2.2]{dinca2021} we need to show that the Jacobian matrix of $\mathcal{H}$ is invertible for $\gamma=0$ for any  $((\mu_{i,K})_{i,K},(\phi_{K})_K)$ .
\begin{lemma}\label{lem:invertible_Jacobian}
	The Jacobian matrix of the map
	\begin{align}\label{eqn:gamma_0_jacobian}
		((\mu_{i,K})_{i},\phi_K)_K\mapsto\left(\left(\mathcal{H}_{i,L}((\mu_{j,K})_K,(\phi_K)_K,0)\right)_{i,L},\left(\mathcal{H}_{\Phi,L}((\mu_{j,K})_K,(\phi_K)_K,0)\right)_L\right)
	\end{align}
	is invertible everywhere.
\end{lemma}
\begin{proof}
	The Jacobian matrix $J_f((\mu_{i,K})_i)$ of $f:\mathbb{R}^n\to\mathbb{R}^n,(\mu_{i,K})_{i=1,\dots,n}\mapsto(u_{i,K})_{i=1,\dots,n}$ is given by
	\begin{align*}
		\frac{\partial u_{i,K}}{\partial \mu_{j,K}} = \delta_{i,j}\frac{\exp(\mu_{i,K})}{1+\sum_{l=1}^n\exp(\mu_{l,K})}-\frac{\exp(\mu_{i,K})\exp(\mu_{j,K})}{\left(1+\sum_{l=1}^n\exp(\mu_{l,K})\right)^2}=\delta_{i,j}u_{i,K}-u_{i,K}u_{j,K},
	\end{align*}
	with $\delta_{i,j} = 1$ if $i=j$ and $0$ otherwise.
	For any $v=(v_i)_{i=1}^n\in\mathbb{R}^n$ we can use the Cauchy-Schwarz inequality $(\sum_i a_i b_i)^2 \le \sum_i a_i^2\sum_i b_i^2$ with $a_i=\sqrt{u_{i,K}}$ and $b_i=\sqrt{u_{i,K}}v_i$ to derive
	\begin{align*}
		v^T J_f((\mu_{i,K})_i)v
		&=\sum_{i=1}^n u_{i,K}v_i^2-\left(\sum_{i=1}^n v_iu_{i,K}\right)^2
		\ge\sum_{i=1}^n u_{i,K}v_i^2-\left(\sum_{i=1}^n u_{i,K}v_i^2\right)\left(\sum_{i=1}^n u_{i,K}\right)\\
		&=\left(1-\sum_{i=1}^n u_{i,K}\right)\sum_{i=1}^n u_{i,K}v_i^2.
	\end{align*}
	Since $1-\sum_{i=1}^n u_{i,K}>0$, the Jacobian matrix $J_f((\mu_{i,K})_i)$ is positive definite everywhere.
	Defining the linear map
	\begin{align*}
		g:\mathbb{R}^{n\#\mathcal{M}}\to\mathbb{R}^{n\#\mathcal{M}},(\mu_{i,K})_{i,K}\mapsto \left(\sum_{S\in\mathcal{T}_K} |S|\nabl \mu_{i,\mathcal{T}}(S)\cdot\nabl e_K(S)\right)_{i,K}
	\end{align*}
	we can write that
	\begin{align*}
		\sum_{i=1}^n \sum_{K\in \mathcal{M}} \mu_{i,K} g((\mu_{j,L})_{j,L})_{i,K} = \sum_{i=1}^n \sum_{S\in\mathcal{T}}|S||\nabl \mu_{i,\mathcal{T}}(S)|^2,
	\end{align*}
	which proves that $g$ has only non-negative eigenvalues. Therefore, the Jacobian matrix of
	\begin{align}\label{eqn:block1}
		(\mu_{i,K})_{i,K}\mapsto \left(\mathcal{H}_{i,K}((\mu_{j,L})_{j,L},(\phi_L)_L,0)\right)_{i,K}=\left( |K| \frac{f((\mu_{j,K})_{j})_i -\left(\frac{M_i^{k}}{|\Omega|}\right)}{\tau^{k}} +g((\mu_{j,L})_{j,L})_{i,K}\right)_{i,K}
	\end{align}
	is invertible everywhere.
	Additionally, the Jacobian matrix of the linear map
	\begin{align}\label{eqn:block2}
		(\phi_K)_K\mapsto \left(\mathcal{H}_{\Phi,L}((\mu_{j,K})_K,(\phi_K)_K,0)\right)_L
	\end{align}
	is invertible everywhere (here we use Assumption \eqref{eq:poincare}).
	Since the Jacobian matrix of \eqref{eqn:gamma_0_jacobian} is a block diagonal matrix with the Jacobian matrix of \eqref{eqn:block1} and \eqref{eqn:block2} respectively on the diagonal, the Jacobian matrix of \eqref{eqn:gamma_0_jacobian} is invertible everywhere.
\end{proof}

\begin{proof}[\bf Proof of Theorem \ref{theo:existsolscheme}]
	We prove the claim by induction on $k$.\\
	For $k=0$ the solution is given by the initial conditions.\\
	Suppose there exists a solution for the time step $k-1$.
	With Lemma \ref{lemma:conservation_of_mass} the solution to time step $k-1$ is non-empty in the sense of \eqref{eqn:initial_conditions_not_empty}.
	To show that there exists a solution to the time step $k$ we use a topological degree argument.
	Since $(\phi_K)_K$ is uniquely determined by a linear equation with bounded right hand side, we obtain that $\phi$ is uniformly bounded by some constant $\overline{C}_\phi$.
	With Lemma \ref{lem:entropy_bout_top}, we obtain, that the map $\mathcal{H}$ has no zeros on the boundary of $[-R,R]^{(n+1)\# V}\times[0,1]$ with $R$ given by
	\begin{align*}
		R:=\max\left\{ \overline{C},\overline{C}_\phi\right\}+1.
	\end{align*}
	Furthermore, a solution for $\gamma=0$ for the equations \eqref{eq:degtop} is given by
	\begin{align*}
		\mu_{i,K} = \log\left(\frac{M_i^{k-1}}{M_0^{k-1}}\right)\quad\forall K\in\mathcal{M},\ \forall i=1,\dots,n.
	\end{align*}
	With Lemma \eqref{lem:invertible_Jacobian} and \cite[Lemma 1.2.2]{dinca2021} follows that the topological degree is non-zero for $\gamma=0$.
	With the homotopy invariance of the topological degree follows that the equations \eqref{eq:degtop} have at least one solution.
	Thus, there exists a solution to the scheme \eqref{eq:scheme} for the time step $k$, which is  \eqref{eq:degtop} with $\gamma = 1$.
\end{proof}

\medskip

{\bf Acknowledgments.} The authors thank Cl\'ement Canc\`es for fruitful suggestions.

\end{document}